\documentclass[a4paper,10pt]{article}
\usepackage{microtype}
\usepackage{amsmath}
\usepackage{amssymb}
\usepackage{amsthm}
\usepackage[T1]{fontenc}
\usepackage{lmodern}
\usepackage[utf8]{inputenc}

\usepackage{setspace}
\usepackage{graphicx}
\usepackage{color}
\usepackage{cite}
\usepackage{lscape}
\usepackage{pdfpages}

\usepackage{caption}
\usepackage{subcaption}
\usepackage{bbm}
\usepackage{enumitem}
\usepackage[usenames,dvipsnames]{xcolor}
\usepackage[color=SeaGreen]{todonotes} 
\usepackage{comment}
\usepackage{mathrsfs}
\usepackage{algorithm}
\usepackage{algpseudocode}
\usepackage[hyperfootnotes=false]{hyperref} 

\hypersetup{colorlinks=false}
\usepackage{multirow}
\usepackage{longtable}
\usepackage{dcolumn}
\usepackage{booktabs}
\usepackage{mathdots}
\usepackage{hhline}
\usepackage[space]{grffile}
\usepackage{bm}
\newcommand{\legendre}[2]{\left(\frac{#1}{#2}\right)}
\newcommand{\bmh}[1]{\bm{\hat{#1}}}

\specialcomment{commentprimes}{}{}
\specialcomment{commentintegers}{}{}
\usepackage{geometry}
\usepackage{enumitem}

\usepackage{amsfonts}

\newtheorem{thm}{Theorem}[section]
\newtheorem{lem}[thm]{Lemma}

\theoremstyle{definition}

\theoremstyle{remark}
\newtheorem{rmk}{Remark}

\title{Representation of Squares by Nonsingular Cubic Forms}
\author{Lasse Grimmelt, Will Sawin      \\ }
\begin{document}
\maketitle

\begin{abstract}
We prove an asymptotic formula for the number of representations of squares by nonsingular cubic forms in six or more variables. The main ingredients of the proof are Heath-Brown's form of the Circle Method and various exponential sum results. The depth of the exponential sum results is comparable to Hooley's work on cubic forms in nine variables, in particular we prove an analogue of Katz' bound.
\end{abstract}

\section{Introduction}
Let $n\geq 6$ and $C\in \mathbb{Z}[X_1,\ldots,X_n]$ be a non-singular cubic form. In this paper we prove asymptotics related to the number of integral solutions of
\begin{align}\label{goal}
C(x_1,\ldots,x_n)=y^2
\end{align}
in an expanding region. We denote by $\Upsilon(X)$ the the number of solutions $(x_1,\ldots,x_n,y)\in \mathbb{Z}^{n+1}$ of \eqref{goal} weighted by some smooth weight. It is defined precisely in \eqref{Updef}. Our main result is the following asymptotics for $\Upsilon(X)$.
\begin{thm}\label{Thm1.1} Assume we are given $C(x_1\ldots,x_n)$ as above with $n\geq 6$. There exists a $\delta>0$, such that
$$\Upsilon(X)=X^{n-\frac{3}{2}}\Bigl(\mathcal{J}\mathfrak{S}+O((\log X)^{-\delta})\Bigr).$$
Here $\mathcal{J}>0$ is the singular integral and $\mathfrak{S}$ is the singular series. Furthermore we have $\mathfrak{S}>0$, if there exist $p$-adic solutions of \eqref{goal} for all $p$. In particular, if $C$ is primitive then $\mathfrak{S}>0$.
\end{thm}

Note that the existence of non-trivial solutions of \eqref{goal} is obvious in many cases. If we set $a$ to be the sum of all coefficients of $C$ and assume $a\neq 0$, then we have the nonzero solution
\begin{align*}
C(a,\ldots,a)=(a^2)^2.
\end{align*}

The question of whether the nonsingular homogeneous cubic indeterminate equation
\begin{align}\label{kubk}
C(x_1,\ldots, x_{n})=0
\end{align}
has a non-trivial solution over the rational integers has been answered in different cases. The existence of local solutions in every $p$-adic field $\mathbb{Q}_p$ is necessary and it is thought that this condition is sufficient, provided $n$ is not too small. Heath-Brown showed an asymptotics for the number of solutions of \eqref{kubk} if $n\geq 10$, see \cite{hb2}. Extending this result, Hooley considered the more challenging case $n=9$ in a series of papers \cite{h9}\cite{h92}\cite{h93}. For our task mostly relevant are the first and third part of this series. In the first Hooley succeeded in showing the existence of non-trivial solutions of $\eqref{kubk}$, in the third an asymptotic was established. We note that if one assumes the truth of the Riemann hypothesis for a certain kind of Hasse-Weil L-functions, also the $n=8$ case can be handled. This was proved by Hooley in \cite{h8}. These results form the starting point for the proof of Theorem \ref{Thm1.1}. Our methodology is most comparable with the considerations of \cite{h93} on cubic forms in $9$ variables in the following aspects. We deduce an asymptotics and require algebraic input  of similar strength to estimate exponential sums.

A natural generalisation of the homogeneous equation \eqref{kubk} is to consider non homogeneous and possibly singular cubic polynomials. This was done, for example, by Browning and Heath-Brown in \cite{bhb}. In that work it is showed that the existence of $p$-adic solutions is sufficient to have infinitely many integral solutions of a cubic polynomial equation, provided the number of indeterminates is at least $11+\nu$. Here $\nu$ is the dimension of the singular locus of the leading cubic form with $\nu=-1$ if it is nonsingular. This result is applicable to \eqref{goal}, but gives a weaker result, since it does not use the structure of the separate quadratic variable.

As is common for results of the type as Theorem \ref{Thm1.1}, we employ a variant of the Hardy-Littlewood Circle Method. More precisely we use the delta-method that was developed by Heath-Brown in \cite{hb1}. This variant has been applied by Hooley in the above mentioned result \cite{h8} and we use the initial analysis of the method provided in that paper. Theorem \ref{Thm1.1} could be proved by a more classical form of the circle method, as used in \cite{h9}\cite{h93}. However, the application of Heath-Brown's new form reduces the technical complexity of the proof. After stating preparatory results and introducing some notation in the next section, we apply the circle method in section \ref{3sec2}.

We set
\begin{align} \label{Ziel}
f(x_1,\ldots, x_{n+1})=C(x_1,\ldots,x_{n})-x_{n+1}^2,
\end{align}
and denote by bold letters $n+1$ dimensional vectors. The main ingredient for the proof of Theorem \ref{Thm1.1} are estimates of several objects related to the exponential sum
\begin{align}\label{Qdef}
Q(\bm{m},k)= {\sum_{h(k)}}^*\sum_{\bm{l}(k)}e_k(hf(\bm{l})+\bm{m}.\bm{l}).
\end{align} 
We note that for odd $k$  the evaluation of the quadratic Gauss sum over $l_{n+1}$ gives us
\begin{align} \label{lnausgewertet}
|Q(\bm{m},k)|=|\sqrt{k}{\sum_{h(k)}}^*\legendre{h}{k}e_k(\overline{4h}m_{n+1}^2)\sum_{\bmh{l}(k)}e_k(hC(\bmh{l})+\bmh{m}.\bmh{l})|,
\end{align}
where we wrote $\bmh{l}=(l_1,\ldots,l_n)$ and similar for $\bmh{m}$. The exponential sum that plays a central role in Hooley's work on cubic forms is 
\begin{align*}
Q'_r(\bmh{m},k)={\sum_{h(k)}}^*e_k({\overline{h}r})\sum_{\bmh{l}(k)}e_k(hC(\bmh{l})+\bmh{m}.\bmh{l}).
\end{align*}
The estimates we require are closely related to the ones appearing in Hooley's papers \cite{h9} \cite{h93}. Our goal is to show that these estimates still hold with the additional multiplicative character 
\begin{align*}
\legendre{h}{k}.
\end{align*}

In subsection \ref{3sec3.1} we consider the first type of bound, namely a pointwise estimate for $Q(\bm{m},k)$ in the case where $k$ is prime or a square of a prime. The basis for estimating $Q(\bm{m},p)$ is Deligne's proof of the Riemann Hypothesis for varieties over finite fields, applied similarly as in Heath-Brown's \cite{hb2}. After summing trivially over $h$ in \eqref{Qdef} an application of Deligne's result gives for any $\bm{m}$
\begin{align*}
Q(\bm{m},p)\ll p^{\frac{1}{2}n+\frac{3}{2}}.
\end{align*}
To prove Theorem \ref{Thm1.1} we need to improve upon this bound for most $\bm{m}$ and show that we can also sum with cancellation over $h$. In our case this requires more delicate results than would follow from a direct application of Deligne's work. We show the bound 
\begin{align} \label{sqroh}
Q(\bm{m},p)\ll p^{\frac{1}{2}n+1}
\end{align}
in the following cases. 
\begin{itemize}
    \item When $p\nmid m_{n+1}$.
    \item When $p|m_{n+1}$ and $\bmh{m}$ is not the simultaneous zero mod $p$ for certain polynomials $F_1, F_2 \in \mathbb{Z}[X_1,\ldots,X_n]$.
\end{itemize}
When $p\nmid m_{n+1}$, \eqref{sqroh} follows from an application  Katz' \cite{katzbhb}. In the second case $F_1$ is the polynomial such that $F_1=0$ is the dual variety to $C=0$. The polynomial $F_2$ is chosen to vanish on a certain special subvariety of the dual variety, which is constructed by a complicated induction-on-dimension argument in \cite{K-L}, but certainly contains the singular locus of the dual variety to $C=0$. We obtain \eqref{sqroh} the second case by combining results of Katz \cite{katzmix} \cite{katz}, and Fouvry and Katz \cite{fok}. 

We prove a similar estimate for $Q(\bm{m},p^2)$, except that only $F_1$ plays a role, see Lemma \ref{primesquare}. The estimate in this prime square case can be easily reduced to the homogeneous case.

In subsection \ref{3sec3.2} we consider a second type of estimate, namely average bounds for 
\begin{align}\label{Qav}
\sum_{|m_1|,\ldots ,|m_n|\leq y}|Q(\bm{m},k)|,
\end{align}
 for $k$ being cube-full and there may be an additional restriction on $\bm{m}$. In particular we prove one general estimate for \eqref{Qav} and one that is restricted to the case $m_{n+1}=0$ and the summation ranging over zeroes of either $F_1$ or $F_1$ and $F_2$. The second and more delicate case is necessary to compensate for the weaker pointwise bounds for $Q(m,p)$ and $Q(m,p^2)$ for certain $\bm{m}$. To obtain the necessary saving, we both use additional summation with cancellation that is possible since $m_{n+1}=0$ and the sparseness of integral vectors $\bmh{m}$ for which either one or both of $F_1$ and $F_2$ vanish. 

In subsection \ref{3sec3.3} we consider a third and final type of exponential sum related estimate. Instead over averaging $Q(\bm{m},k)$ over $m_1,\ldots,m_n$ lying in some bounded area of $\mathbb{Z}^n$, we consider the average over a complete system of representatives modulo $k$
\begin{align*}
\sum_{m_1,\ldots, m_n (k)}|Q(\bm{m},k)|.
\end{align*}
As in \cite{h9} we approach these sums in two ways. First we use an elementary second moment bound together with Cauchy's inequality. Afterwards we deal with the more delicate case of showing that these results can be improved for $k$ lying in some subset of the primes. As in the homogeneous case this improvement is based on Katz' work \cite{katz}. Katz's method shows a nontrivial bound as long as the A-number, which calculates the dimension of a certain cohomology group, is not equal to $1$. He gives a formula for the A-number in terms of the Euler characteristics of a perverse sheaf derived from the exponential sum. In section \ref{3secapp}, we find the perverse sheaf and then calculate these Euler characteristics, finding the A-number and verifying it is greater than $1$.
In section \ref{additional results} we import and prove some elementary additional results and finally combine everything to prove Theorem \ref{Thm1.1} in section \ref{3sec5}. As in the homogeneous case the final estimation is done separately for different cases that stem from the different strength of bounds for $Q(\bm{m},p)$ for certain $\bm{m}$.

\begin{rmk} One may ask whether Theorem \ref{Thm1.1} fits into the general framework of Manin's conjecture, i.e., is equivalent to a statement about counting rational points of bounded height on a Fano variety. This is not possible, because there does not exist a double cover of projective space ramified only at a hypersurface of degree $3$ - in fact, there only exist such double covers for hypersurfaces of even degree. Instead, we can express it as a statement of counting rational points on a stack. We can construct the stack as a hypersurface in the (stacky) weighted projective space with weights $2,\dots,2,3$ defined by the weighted-homogenous equation $C(x_1,\dots,x_n) =y^2$ or, equivalently, as the quotient of the scheme of affine nonzero solutions of that equation by the natural action of $\mathbb G_m$. We expect that the height $\sum_{i=1}^n |x_i|^2$ we use to count will occur as the height associated to the line bundle $\mathcal O(1)$ as part of the general formalism in upcoming work of Jordan Ellenberg, Matt Satriano, and David Zureick-Brown defining heights on stacks using vector bundles.\end{rmk}

\begin{rmk} It may be possible to prove \ref{Thm1.1} for $n=5$ conditional on the analytic continuation and Riemann hypothesis for Hasse-Weil L-functions, following the method of Hooley in \cite{h8}. However, we have not checked any of the details of this. First, one needs enable summation over $k$ with cancallellation by showing a suitable version of \cite[Lemma 7]{h8}. Following Hooley's argument, this then most likely leads to the $L$-functions of the double covers of $\mathbb P^{n-1} = \mathbb P^4$ ramified at the cubic hypersurface defined by $C$ and at an arbitrary hyperplane. In addition, an Euler characteristic calculation shows that the local factors of this Hasse-Weil L-function are complicated enough (degree $16$) that even proving analytic continuation seems out of reach at the moment.\end{rmk}

\begin{rmk}
Instead of reducing the number of variables, another direction to extend Theorem \ref{goal} is to consider
\begin{align*}
C(x_1,\ldots,x_6)-y^2=N
\end{align*}
for some fixed $N$. The application of the delta-method is basically also possible in that case. The difficulty lies in showing suitable exponential sum bounds. However, one probably should first consider the easier problem
\begin{align}
    C(x_1,\ldots,x_9)=N
\end{align}
and extend Hooley's \cite{h9} and \cite{h93} to cover this case. 
\end{rmk}

\section{Preparation} \label{3sec2}
We now fix the following notation. Assume we are given $n\geq 6$ and a cubic form $C\in\mathbb{Z}[X_1\ldots X_n]$ that is non-singular over $\mathbb{C}$. We write $\bm{x}=(x_1,\ldots,x_{n+1})$, $\bmh{x}=(x_1,\ldots,x_n)$ and want to count weighted zeroes of
\begin{align*}
f(\bm{x})=C(\bmh{x})-x_{n+1}^2.
\end{align*}
Before we introduce the weight $\Gamma$ that appears in Theorem \ref{Thm1.1} we start with preparatory observations as done in Hooley's \cite{h9} and \cite{h8}. They consist of two parts. The first one is the proof that there are real zeros of a non-singular cubic form at which the Hessian of the form is not zero. In our case a more simple treatment would also suffice but for sake of analogy we use the existing result. 

There is the matrix 
\begin{align*}
\bm{M}(\bm{x}):=\frac{\partial^2 f}{\partial x_i \partial x_j} \text{  } (1\leq i,j\leq n+1)
\end{align*}
whose determinant 
\begin{align*}
H(\bm{x}):=|\bm{M}(\bm{x})|
\end{align*}
is the Hessian covariant of $f(\bm{x})$.
If we denote by
\begin{align*}
\bm{M}_C(\bmh{x}):=\frac{\partial^2 f}{\partial x_i \partial x_j} \text{  } (1\leq i,j\leq n)
\end{align*}
and
\begin{align*}
H_C(\bmh{x}):=|\bm{M}_C(\bmh{x})|
\end{align*}
the contribution of the cubic form to the above objects, then we have

\begin{align}\label{Hesseumrechnung}
H(\bm{x})=2H_C(\bmh{x}).
\end{align}

\begin{lem}
There is a point $\bm{a}=(\bmh{a},0)\in \mathbb{R}^{n+1}$ fulfilling
\begin{align*}
H(\bm{a})\neq 0
\end{align*}
and
\begin{align*}
f(\bm{a})=C(\bmh{a})-0^2=0.
\end{align*}
\end{lem}
\begin{proof}
The result follows directly from Lemma 1 of \cite{h9} applied to $C$ and inserted in \eqref{Hesseumrechnung}.
\end{proof}

The second important type of result at this stage are estimates of the number of integral zeros related to the bad cases in the exponential sum estimates. As mentioned before, to compensate for weaker bounds for $Q(\bm{m},p)$ and $Q(\bm{m},p^2)$ we use the sparseness of integral vectors $\bmh{m}$ for which the polynomials $F_1$ and $F_2$ vanish. The exact definition of $F_1$ and $F_2$ is stated in the proof of Lemma \ref{primecase}, we here only need to know that $F_2$ is not an integral multiple of $F_1$. Let 
\begin{align}\label{N1def}
N_1(y,\bmh{r},k):=\# \{\bmh{m}\in \mathbb{Z}^{n} : ||\bmh{m}||\leq y, \bmh{m}\equiv \bmh{r} (k), F_1(\bmh{m})=0  \}
\end{align}
and let further
\begin{align}\label{N2def}
N_2(y,\bmh{r},k):=\# \{\bmh{m}\in \mathbb{Z}^{n} : ||\bmh{m}||\leq y, \bmh{m}\equiv \bmh{r} (k), F_1(\bmh{m})=0, F_2(\bmh{m})=0 \}.
\end{align}
We bound the number of bad vectors with the following lemma.
\begin{lem}\label{Nlemma}
It holds that
\begin{align*}
N_1(y,\bmh{r},k)\ll\Bigl(\frac{y}{k}+1 \Bigr)^{n-1}
\end{align*}
and
\begin{align*}
N_2(y,\bmh{r},k)\ll\Bigl(\frac{y}{k}+1 \Bigr)^{n-2}.
\end{align*}
\end{lem}
\begin{proof} Because $F_1$ is irreducible and $F_2$ is not a multiple of $F_1$, the vanishing set of $F_1$ and $F_2$ is a Zariski closed subset of codimension $2$ in $\mathbb A^n$.  This follows from the general fact that a Zariski closed subset of codimension $c$ in $\mathbb A^n$ contains $\ll \Bigl(\frac{y}{k}+1 \Bigr)^{n-c}$ integer points whose coefficients are $\leq y$ congruent to $k$ mod $m$. This can be checked by induction on $c$ and $m$. For all but $O(1)$ numbers, restricting the first coordinate to that value produces a Zariski closed subset of codimension $c$. For the remaining $O(1)$ numbers, it is codimension $c-1$. Hence the bound for $(n,c)$ is at most the total number of values for the first coordinate, which is $2\frac{y}{k}+1$, times the bound for $(n-1,c)$ plus $O(1)$ times the bound for $(n-1,c-1)$, both of which are $\Bigl(\frac{y}{k}+1 \Bigr)^{n-c}$ by induction, giving the stated bound.

\end{proof}

\subsection{Application of the delta-method}
We start our analytical treatment by defining the weight we want to count the solutions of \eqref{Ziel} with as
\begin{align*}
\Gamma(\bm t)=\prod_{1\leq i\leq n+1}\gamma(t_i),
\end{align*}
where
\begin{align*}
\gamma(t)=\begin{cases} e^{-2/(1-t^2)}, & \text{if } |t|< 1\\ 0, & \text{if } |t|\geq 1.
\end{cases}
\end{align*}
We apply the considerations of the previous section to choose a point $\bm{a}$ that is a sufficiently large scalar multiple $\lambda \bm{a'}$ of a given real point $\bm{a'}$ such that $H(\bm{a'})\neq 0$ and $f(\bm{a'})=C(\bmh{a}')=0$.
The smoothed counting function of Theorem \ref{Thm1.1} is given by
\begin{align}\label{Updef}
\Upsilon(X)=\sum_{f(\bm{x})=0} \Gamma((X^{-1}\bmh{x},X^{-\frac{3}{2}}x_{n+1})-\bm{a}).
\end{align}
It counts with a certain weight the number of integral solutions of $C(\bmh{x})=x_{n+1}^2$ in a hyperparallelepiped with side length 2X in the first $n$ coordinates and length $2X^{\frac{3}{2}}$ in the last one.

To translate the counting problem into the realm of analysis, we apply Heath-Brown's delta-method as given in Theorem 2 of \cite{hb1}. We get
\begin{align}\label{UP1}
\Upsilon(X)=\frac{c_N}{N^2}\sum_{\bm{m}}\sum_{k=1}^\infty k^{-n-1}Q(\bm{m},k)I_k^{(0)}(\bm{m})
\end{align} 
for all $N$ exceeding $1$, where
\begin{align*}Q(\bm{m},k)= {\sum_{h(k)}}^*\sum_{\bm{l}(k)}e_k(hf(\bm{l})+\bm{m}.\bm{l})
\end{align*}
is the important exponential sum, 
\begin{align*}I_k^{(0)}(\bm{m})=\int_{\mathbb{R}^{n+1}}\Gamma((X^{-1}\bmh{x},X^{-\frac{3}{2}}x_{n+1})-\bm{a})h\Bigl(\frac{k}{N},\frac{f(\bm{x})}{N^2}\Bigr)e_k(-\bm{m}.\bm{x})d\bm{x},
\end{align*}
and the objects $h(x,y)$ and $c_N$ appear explicitly in \cite{hb1}.
Furthermore by \cite[Theorem 1]{hb1} we have for any $A>0$
\begin{align}\label{c_N}
c_N=1+O_A(N^{-A}).
\end{align} 
Our considerations differ in one point from the case of a cubic form as considered in Hooley's application of the delta-method in \cite{h8}. The last variable is weighted differently which is a natural consequence of it appearing only in degree two and this changes the following transformation. We fix $N=X^{\frac{3}{2}}$ and transform $I_k^{(0)}$ into
$$
X^{n+\frac{3}{2}}\int_{\mathbb{R}^{n+1}}\Gamma(\bm{x}-\bm{a})h\Bigl(\frac{k}{N},f(\bm{x})\Bigr)e_k(-\bm{m}.(X\bmh{x},X^{\frac{3}{2}}x_{n+1}))d\bm{x}=X^{n+\frac{3}{2}}I_k(\bm m),
$$
say.

 By Lemma 4 of \cite{hb1} there exists a constant $A_1$ such that \begin{align*}
 I_k(\bm{m})=0
\end{align*}  
for $k>A_1N$. In \eqref{UP1} we now split in the remaining range $k\leq A_1 N$ the summation over $\bm{m}$ in \eqref{UP1} into different depending on the conditions $m_{n+1}=0$, $F_1(\bmh{m})=0$, and $F_2(\bmh{m})=0$. As much of this paper, this split is motivated by the worse estimates for $Q(\bm{m},p)$ and $Q(\bm{m},p^2)$ in different cases. We recall the notation $\bm{m}=(m_1,\ldots,m_{n+1})$, $\bmh{m}=(m_1,\ldots,m_{n})$ and since the case $m_{n+1}=0$ is of particular interest, we also write $(\bmh{m},0)=(m_1,\ldots,m_n,0)$. We continue with \eqref{UP1} to get
\begin{align}
&\Upsilon(X)= \nonumber\\
=& c_{N} X^{n-\frac{3}{2}}\sum_{\bm{m}}\sum_{k\leq A_1 N} k^{-n-1}Q(\bm{m},k)I_k(\bm{m}) \nonumber \\ 
=&c_{N}X^{n-\frac{3}{2}}\Bigl(\sum_{k\leq A_1 N}k^{-n-1}Q(\bm{0},k)I_k(\bm{0}) + \sum_{k\leq A_1 N}\sum_{\substack{\bmh{m}\neq 0\\F_1(\bm{\hat{m}})\neq0}}k^{-n-1}Q((\bmh{m},0),k)I_k(\bm{m})\nonumber \\
&+\sum_{k\leq A_1 N}\sum_{\substack{\bm{m}\\m_n\neq0}}k^{-n-1}Q(\bm{m},k)I_k(\bm{m})+\sum_{k\leq A_1 N}\sum_{\substack{\bmh{m}\neq 0\\ \substack{F_1(\bm{\hat{m}})=0\\ F_2(\bmh{m})=0}}}k^{-n-1}Q((\bmh{m},0),k)I_k(\bm{m})\nonumber \\
&+\sum_{k\leq A_1 N}\sum_{\substack{\bmh{m}\neq 0\\ \substack{F_1(\bm{\hat{m}})=0\\ F_2(\bmh{m})\neq 0}}}k^{-n-1}Q((\bmh{m},0),k)I_k(\bm{m})\nonumber\Bigr)\nonumber \\ 
\label{UPdef} =&c_{N}X^{n-\frac{3}{2}}\Bigl(\Upsilon_1(X)+\Upsilon_2(X)+\Upsilon_3(X)+\Upsilon_4(X)+\Upsilon_5(X)\Bigr),
\end{align}
say. Here $F_1$ and $F_2$ are the polynomials appearing in \eqref{N1def} and \eqref{N2def}. We expect that the largest contribution arises from $\Upsilon_1(X)$ and that this is asymptotically given by a product of the \emph{Singular Series} $\mathfrak{S}$ and the \emph{Singular Integral} $\mathcal{J}$. This product is bounded and, given the existence of $p$-adic and real solutions, does not vanish. Throughout Section \ref{3sec5} we show that the other terms are $O((\log X)^{-\delta})$ and this proves the asymptotics stated in Theorem \ref{Thm1.1}. 

\subsection{Bounds for the Integral}
In Section 5 of \cite{h8} Hooley analyses the behaviour of $I_k(\bm{m})$ and $\frac{\partial I_k(\bm{m})}{\partial k}$ in the case of a cubic form. The estimates related to $\frac{\partial I_k(\bm{m})}{\partial k}$ are necessary for summation with cancellation over $k$, which we do not intend to do. They could not, at least for all $\bm{m}$, be applied in our case, but the calculations for $I_k(\bm{m})$ and the use of Lemma 7 of \cite{h9} do not depend on $f$ being homogeneous. We only need to consider the different substitution by defining
$$
\lVert \bm{m} \rVert_X := \lVert(\hat{\bm{m}},X^{\frac{1}{2}}m_{n+1}) \rVert .
$$
We then can set for some $M>0$
\begin{align}\label{Jdef}
&J(\bm{m},Y)=\\
=&\begin{cases} 1 & \!\!\text{if } \lVert \bm{m} \rVert_X\leq Y/X, \\
\log^{n+1}(2X \lVert \bm{m} \rVert_X/Y)(Y/X\lVert \bm{m} \rVert_X)^{\frac{1}{2}n-\frac{1}{2}} &\!\! \text{if } Y/X<\lVert \bm{m} \rVert_X\leq N/X, \\
\log^{n+1}(2X \lVert \bm{m} \rVert_X/Y)(Y/X\lVert \bm{m} \rVert_X)^{\frac{1}{2}n-\frac{1}{2}}(N/X\lVert \bm{m} \rVert_X )^M &\!\! \text{if } N/X<\lVert \bm{m} \rVert_X.
\end{cases} \nonumber
\end{align}
With this notation, Hooley's proof can be applied and gives us the following estimates  for $I_k(\bm{m})$.
\begin{lem} \label{Integral}
Let $M>0$. For $\frac{1}{2}Y\leq k <Y$ and $Y\leq A_1 N= A_1 X^{\frac{3}{2}}$ it holds that
\begin{align*}
I_k(\bm{m})=O_M(J(\bm{m},Y)).
\end{align*}

\end{lem}
It is an advantage of the delta-method that this is sufficient analytic input for proving Theorem \ref{Thm1.1}

\section{Exponential Sums} \label{3sec3}
We now state and prove estimates for exponential sums that will be later combined with the integral bounds to bound $\Upsilon_i$ for $2\leq i \leq 5$. The results are analogues of Lemma 10, 11, 12, equations (95), (96), (103), (107), (108), and (110) of \cite{h9} and Lemma 60 of \cite{h93}. 

\subsection{Preparation, Prime, and Prime square case}\label{3sec3.1}
We now start with the study of the exponential sums. Recalling
\begin{align*}Q(\bm{m},k)= {\sum_{h(k)}}^*\sum_{\bm{l}(k)}e_k(hf(\bm{l})+\bm{m}.\bm{l}),
\end{align*} where $f$ is a cubic form minus a square as defined in \eqref{Ziel}. In what follows there will be a difference in the effect of $m_1,\ldots, m_n$ and of $m_{n+1}$. As described in the introduction $m_{n+1}$ behaves similar to $r$ in \cite{h9} and \cite{h93}. In particular the cases $m_{n+1}=0$ and zeroes of $F_1$ and $F_2$ need separate consideration.

The exponential sum $Q(\bm{m},k)$ behaves pseudo-multiplicatively in $k$, i.e. we have for $(k,k')=1$ 
\begin{align}\label{pseudomult}
Q(\bm{m},kk')=Q(\overline{k'}\bm{m},k)Q(\overline{k}\bm{m},k').
\end{align}
As before we write $(\bmh{m},0)=(m_1\ldots,m_n,0)$ and note that  $|Q((\bmh{m},0),k)|$ is multiplicative in $k$ as for $(a,k)=1$
\begin{align} \label{betragmult}
|Q(a(\bmh{m},0),k)|=|Q((\bmh{m},0),k)|.
\end{align}
These properties allow us to separately deal with the cases of $k$ being either prime or the square of a prime, which will both be estimated in this subsection.
The central idea behind all of our estimates related to these exponential sum is that they should behave as $\sqrt{k}$ times an estimate in the same case of a cubic form in $n$ variables.

The first result we require is a point-wise estimate of $Q(\bm{m},p)$ depending on $m_{n+1}$, $F_1(\bm{m})$ and $F_2(\bm{m})$. For details of some results on sheaves that are used in the proof, see subsection \ref{3secapp}.

\begin{lem}\label{primecase}
There exists a polynomial $F_2(\bmh{m})$ that is not a multiple of $F_1(\bmh{m})$, such that we have
\begin{align*}
Q(\bm{m},p)=\begin{cases} 
O(p^{\frac{1}{2}n+1}(m_{n+1},p))^{\frac{1}{2}} & \text{if } m_{n+1}\neq 0, \\
O(p^{\frac{1}{2}n+1}(F_1(\bm{\hat{m}}),F_2(\bmh{m}),p))^{\frac{1}{2}} & \text{if } m_{n+1}=0.
\end{cases}
\end{align*}
\end{lem}
\begin{proof}
 The estimate in the homogeneous case was proved by Heath-Brown in Section 5 of \cite{hb2}. We start by noting that it is enough to only consider those primes for which neither $p|6$ nor is $C$ singular over $\mathbb{F}_p$, as the finitely many exceptions can be dealt with by increasing the implied constants. We recall the evaluation of the sum over $l_{n+1}$ as given in \eqref{lnausgewertet}
 \begin{align*}
|Q(\bm{m},p)|=|\sqrt{p}{\sum_{h(p)}}^*\legendre{h}{p}e_p(\overline{4h}m_{n+1}^2)\sum_{\bmh{l}(p)}e_p(hC(\bmh{l})+\bmh{m}.\bmh{l})|,
\end{align*}
We follow Heath-Brown and apply the triangle inequality on the sum over $h$ to get 
\begin{align*}
|Q(\bm{m},p)|\leq \sqrt{p}{\sum_{h(p)}}^*|\sum_{\bmh{l}(p)}e_p(hC(\bmh{l})+\bmh{m}.\bmh{l})|.
\end{align*} 
The innermost sum can be estimated by using a consequence of Deligne's proof of the Riemann-Hypothesis for varieties over finite fields as stated in Theorem 8.4 of \cite{Deligne}. This implies
\begin{align*}
Q(\bm{m},p)=O(p^{\frac{1}{2}n+\frac{3}{2}}).
\end{align*}
We have to show this bound can be improved upon by a factor of size $\sqrt{p}$ when \emph{either}
\begin{itemize}
\item $p\nmid m_{n+1}$ \emph{ or}
\item $m_{n+1}=0$ and $p$ does not divide at least one of $F_1(\bmh{m})$ and $F_2(\bmh{m})$.
\end{itemize}
The additional character seems to make the elementary reduction Heath-Brown uses and further direct application of Deligne's result not possible. We instead use results from Katz \cite{katzbhb} or Fouvry and Katz \cite{fok}.

In the case that $p\nmid  m_{n+1}$ we want to apply Corollary 8.2 of \cite{katzbhb}. To do so, we check that the required conditions are met. In Katz' notation the bounded object is
\begin{align*}
\sum_{h \in  k^\times}\sum_{\bmh{l}}\chi(h)\psi\bigl(h^AC(\bmh{l})+g(\bmh{l})+P_B(1/t)\bigr),
\end{align*}
where $\chi$ is a multiplicative and $\psi$ an additive character of the finite field $k$. Furthermore $C(\bmh{l})$ and $g(\bmh{l})$ are polynomials in $n$ variables of degree $d$ and $e$ respectively. Finally, $P_B$ is a one variable polynomial of degree $B$. In our case we can translate Katz' notation as follows
\begin{align*}
k&= \mathbb F_p\\
\chi(h)&=\legendre{h}{p}\\
\psi(\bullet)&=e_p(\bullet)\\
C(\bmh{l})&=C(\bmh{l})\\
g(\bmh{l})&=\bmh{m}.\bmh{l}\\
P_B(1/h)&=\overline{4}m_{n+1}^2 1/h\\
h^A&=h.
\end{align*}
This means
\begin{align*}
d&=3\\
e&=1\\
A&=1\\
B&=1.
\end{align*}

Corollary 8.2 of \cite{katzbhb} has the hypotheses that $p$ is prime to $AB(A+B)$, $d$ is prime to $p$, $C$ is a Deligne polynomial, and $e< (B/(A+B))d$. (We are also allowed to take $ \deg P_B < B$ if we assume in addition that $C$ is an affine-Dwork regular Deligne polynomial.) The conditions on $p$ are satisfied as soon as $p>3$, which we can ensure by raising the constant if necessary. We have $(B/(A+B))d=\frac{3}{2}$ and as $1<\frac{3}{2}$ the condition on $e$ is met. Because $C$ is homogeneous of degree prime to $p$ and nonsingular, its leading term is nonsingular, and so it is a Deligne polynomial. Thus we obtain the case $m_{n+1}\neq 0$ of Lemma \ref{primecase}.

For $m_{n+1}=0$ we sum over $h$ in \eqref{lnausgewertet} to arrive at
\begin{align*} 
|Q((\bmh{m},0),p)|=p\Bigl|\sum_{\bmh{l}(p)}\legendre{C(\bmh{l})}{p}e_p(\bmh{m}.\bmh{l})\Bigr|.
\end{align*}

We now use the notation of \cite[section 3]{fok}, except that we continue to use the coordinates $\bm{m}$ and $\bm{l}$ instead of their $(x_1,\dots,x_n)$ and $(h_1,\dots,h_n)$. We work in affine $n$-space $\mathbb A^n_{\mathbb Z}$. We take $V$ to be all of $\mathbb A^n_{\mathbb Z}$, of dimension $d=n$. Let $f$ be zero.  The stratification $\mathscr V$ of $\mathbb A^n$ will have three strata: the open set where the cubic form $C$ is nonzero, the vanishing locus of $C$ except for the origin, and the origin. Let $j$ be the inclusion of the open set where $C$ is nonzero, and take $K=j_! \mathcal L_\chi (C)[n] $ for $\chi$ the quadratic character. Then $K$ is adapted to this stratification and, by Lemmas \ref{lisse-pure} and \ref{middle-extension} of subsection \ref{3secapp}, (fibrewise) perverse and pure of weight $n$, hence satisfies assumptions 1) and 2) of \cite[section 3]{fok}. Here assumption 1) is that $K$ is (fibrewise) semiperverse and 2) is that it is mixed of weight $\leq d$.

Then \cite[Theorem 3.1]{fok} produces a stratification of the dual $\mathbb A^n_{\mathbb Z}$ into strata $H_i$ of dimensions $N_i$.  This may require throwing out finitely many primes.

Next we will apply \cite[Theorem 4.4]{fok}. We first verify the assumptions of 4.0.1,4.0.2., and 4.0.3 of \cite{fok}. Assumption 4.0.1 is that $V[1/D]$ is closed in $\mathbb A^n_{\mathbb Z}[1/D]$, smooth, and surjective with geometrically connected fibers over $\mathbb Z[1/D]$. Because $V= \mathbb A^n$, these are all satisfied for $D=1$.  Because $V$ is smooth and surjective, 4.0.1 is satisfied with $D=1$. Assumption 4.0.2 repeats assumptions 1) and 2) of \cite[section 3]{fok} and assumption 4.0.3 strengthens this to $K$ being perverse, geometrically irreducible and pure of weight $n$ fiberwise away from primes dividing $\ell MD$, again by Lemmas \ref{lisse-pure} and \ref{middle-extension} these conditions are satisfied.

By Lemma \ref{a-number} the $A$-number defined in \cite[p. 127]{fok} as the generic rank of the Fourier transform of $K$ is equal to $2^{n-1}$ and thus is nonzero. This verifies the assumptions of \cite[Theorem 4.4]{fok}. 

Hence  for any finite field $k= \mathbb F_p$ of characteristic sufficiently large $6\ell$, for the nontrivial additive character $e_p: \mathbb F_p \to \mathbb C^\times$, and any point $\bm{m}\in H(k) =\mathbb F_p^n$ lying in a strat $H_i$ of dimension $\eta_i$, we have

\[ \left | \sum_{\bm{l} \in \mathbb F_p^n } \left( \frac{C(\bm{l})}{p} \right) e_p  ( \bm{m} \cdot \bm{l}  ) \right| \leq  O(1)  \times p^{\frac{ \sup d, d+n-1 - \eta_i}{2}} .\]
 
In particular, because $d=n$, the bound is $\leq  O( p^{\frac{1}{2}n})$  for any stratum of dimension $\geq n-1$.

Let $Z$ be the union of all the strata of dimension $\leq n-2$, and let $Z'$ be the intersection of $Z$ with the vanishing locus of $F_1$. Because $Z$ has dimension $\leq n-2$ and the vanishing locus of $F_1$ has dimension $d-1$, there exists a polynomial, not a multiple of $F_1$, that vanishes on $Z'$. Let $F_2$ be such a polynomial.

Then for any point $\bm{m} $ except those where $F_1(\bm{m})=F_2(\bm{m})=0$, we either have $F_1(\bm{m}) \neq 0$ or $\bm{m} \not \in Z$. In the first case, the remaining estimate of Lemma \ref{primecase} is a consequence of \cite[Theorem 1.1(1)]{katzmix}. In the second case, it follows from the estimate we have just given.

\end{proof}

The case $k=p^2$ is not considered by Heath-Brown and in the case of cubic forms is \cite[Lemma 11]{h9}. We show the following analogue by reducing it to Hooley's result.
\begin{lem} \label{primesquare}
We have
\begin{align*}
Q(\bm{m},p^2)=\begin{cases} 
O(p^{n+2}(m_n,p)) & \text{if } m_{n+1}\neq 0 \\
O(p^{n+2}(F_1(\bm{\hat{m}}),p)) & \text{if } m_{n+1}=0
\end{cases}
\end{align*}
\end{lem}
\begin{proof}
To prove this we consider \eqref{lnausgewertet}, which was valid for $(2,k)=1$ and note that for $k=p^2$ the Jacobi symbol is always $1$. Consequently we have for all odd primes

\begin{align*}
|Q(\bm{m},p^2)|=p|{\sum_{h(p^2)}}^*e_{p^2}(\overline{4h}m_{n+1}^2)\sum_{\bmh{l}(p^2)}e_{p^2}(hC(\bmh{l})+\bmh{m}.\bmh{l})|.
\end{align*}
This is the same sum that is estimated in \cite[Lemma 11]{h9}, having $\bar{4}m_{n+1}^2$ in the place of Hooley's $r$. We can apply his result and get the stated bound for odd primes. The remaining case can again be absorbed in the implied constants.
\end{proof}

\subsection{Averages with Cube-Full Modulus}\label{3sec3.2}
For cube-full $l_3$ we now look at sums of the type
\begin{align*}
\sum_{\lVert \bm{\hat{m}}\rVert\leq y}|Q(\bm{m},l_3)|
\end{align*} for some $y$. In particular, in Lemma \ref{6.1} we provide a general estimate without any restriction on $\bm{m}$. Afterwards in Lemma \ref{6.2}, we prove two more estimates where the summation is restricted to $m_{n+1}=0$ and those $\bmh{m}$ that are zeroes of $F_1$ or both $F_1$ and $F_2$.

\begin{lem}\label{6.1}
Let $l_3$ be cube-full. We have uniformly in $m_n$ and for all $(a,l_3)=1$
\begin{align*}
\sum_{\lVert \bm{\hat{m}}\rVert\leq y} |Q(a\bm{m},l_3)|\ll l_3^{\frac{1}{2}n+\frac{3}{2}+\epsilon}(y^{n}+l_3^\frac{n}{3}). 
\end{align*}
\end{lem}
\begin{proof}
The result in the homogeneous case is Lemma 14 of \cite{hb2}. The summation over $h$ is done trivially in Heath-Brown's proof, so our additional character disappears in the first step. We have for all $k$ and $(a,k)=1$

\begin{align*}
{\sum_{h(k)}}^*\Bigl|\sum_{\bm{x}(k)}e_k(h(C(\bmh{x})-x_n^2)+\bm{x}.(a\bm{m}))\Bigr|&\leq \sqrt{2k}{\sum_{h(k)}}^*\Bigl|\sum_{\bmh{x}(k)}e_k(hC(\bmh{x})+\bmh{x}.(a\bmh{m}))\Bigr|\\
&=\sqrt{2k}{\sum_{h(k)}}^*\Bigl|\sum_{\bmh{x}(k)}e_k(hC(\bmh{x})+\bmh{x}.\bmh{m})\Bigr|,
\end{align*}
the remaining sum is the same object that Heath-Brown deals with and his argument can be applied to prove the Lemma for squareful and so also for cube-full $k$.

\end{proof}

To compensate for the weaker cases of Lemma \ref{primecase} and \ref{primesquare}, we show the following result. It improves upon Lemma \ref{6.1} by using $m_{n+1}=0$ to do an additional nontrivial summation. Furthermore we employ \ref{Nlemma} to use the relative scarcity of points for which either $F_1(\bmh{m})=0$ or $F_1(\bmh{m})=0$ and additionally $F_2(\bmh{m})=0$. 

\begin{lem}\label{6.2}
Let $k_3l_4$ be cube-full with $k_3=\prod_{p^3||k_3l_4}p^3$, then 
\begin{align*}
\sum_{\substack{\lVert \bm{\hat{m}}\rVert\leq y\\ F_1(\bm{\hat{m}})=0}}|Q((\bmh{m},0),k_3l_4)|\ll  (k_3l_4)^\epsilon(k_3^{\frac{5}{6}n+\frac{4}{3}}l_4^{n+1}+y^{n-1}(k_3l_4)^{\frac{1}{2}n+\frac{5}{3}}).
\end{align*}
Furthermore
\begin{align*}
\sum_{\substack{\lVert \bm{\hat{m}}\rVert\leq y\\ \substack{F_1(\bm{\hat{m}})=0 \\F_2(\bmh{m})=0}}}|Q((\bmh{m},0),k_3l_4)|\ll (k_3l_4)^\epsilon(k_3^{\frac{5}{6}n+\frac{4}{3}}l_4^{n+1}+y^{n-2}(k_3 l_4)^{\frac{1}{2}n+2}).
\end{align*}
\end{lem}
\begin{proof}
The proof of both estimates is done similarly, only in the last step a different case of Lemma \ref{Nlemma} is applied.

We start as in the related case in Chapter $7$ of \cite{hb2}. For the sake of analogy in the notation, instead of looking at a cube-full $k$ we use a square-full $q$ that we write as $q=q_1^2q_2 $ with a squarefree $q_2$, so that $q_2|q_1$. We recall \eqref{Qdef} 
\begin{align*}
Q(\bm{m},q)=\sum_{h(q)}^*\sum_{\bm{l}(q)}e_q(hf(\bm{l})+\bmh{m}.\bmh{l}).
\end{align*}
Since it is eases the notation we start by treating $Q(\bm{m},0)$ for arbitrary $\bm{m}$ before using $m_{n+1}=0$. We transform the sum over $\bm{l}$ by writing $\bm{l}=\bm{L}+q_1q_2\bm{g}$ and applying Taylor expansion
\begin{align*}
\sum_{\bm{l}(q)}e_q(hf(\bm{l})+\bm{m}.\bm{l})&=\sum_{\bm{L}(q_1q_2)}e_q(hf(\bm{L})+\bm{m}.\bm{L})\sum_{\bm{g}(q_1)}e_{q_1}(\bm{g}.(h\nabla f(\bm{L})+\bm{m}))\\
&=q_1^{n+1}\sum_{\substack{\bm{L}(q_1q_2)\\q_1|s\nabla f(\bm{L})+\bm{m}}}e_q(hf(\bm{L})+\bm{m}.\bm{L}).
\end{align*}
We sum now over $h$ and to do so write $h=t+uq_1$. Considering that $(h,q)=1$ if and only if $(t,q_1)=1$, we get

\begin{align}
Q(\bm{m},q)&=q_1^{n+1}{\sum_{t(q_1)}}^*\sum_{\substack{\bm{L}(q_1q_2)\\q_1|t\nabla f(\bm{L})+\bm{m}}}e_q(tf(\bm{L})+\bm{m}.\bm{L})\sum_{u(q_1q_2)}e_{q_1q_2}(uf(\bm{L}))\nonumber \\
&=q_1^{n+2}q_2{\sum_{t(q_1)}}^*\sum_{\substack{\bm{L}(q_1q_2)\\ \substack{q_1|t\nabla f(\bm{L})+\bm{m}\\q_1q_2|f(\bm{L})}}}e_q(tf(\bm{L})+\bm{m}.\bm{L}). \label{Lnsum}
\end{align}
We now fix $m_{n+1}=0$ and recall that we have $f(\bm{L})=C(\bmh{L})-L_{n+1}^2$. The first summation condition of the inner sum becomes
\begin{align*}
&q_1|t\nabla C(\bmh{L})+\bmh{m} \\
&q_1|2tL_{n+1}.
\end{align*} 
For $(q_1,2)=1$ this means $q_1|L_{n+1}$ and as $q_2|q_1$ we have $q_1q_2|L_{n+1}^2$. Using this, the second condition is
\begin{align*}
q_1q_2|C(\bmh{L}).
\end{align*}
The sum over $L_{n+1}$ in \eqref{Lnsum} is consequently of the form
\begin{align*}
\sum_{L_n(q_1q_2), q_1|L_n}e_q(tL_n^2)=\sum_{L_n(q_2)}e_q(tq_1^2L_n^2)=\sum_{L_n(q_2)}e_{q_2}(tL_n^2)\ll q_2^{\frac{1}{2}}.
\end{align*}
Let now $2^\alpha||q_1$ for some $\alpha\geq 1$. We then have 
\begin{align*}
\sum_{L_n(q_1q_2), q_1|2 L_n}e_q(tL_n^2)&=\sum_{L_n(2q_2)}e_q(t2^{2\alpha-2}q_1^2L_n^2)\\&=\sum_{L_n(2 q_2)}e_{4 q_2}(tL_n^2)\\&=\frac{1}{2}\sum_{L_n(4 q_2)}e_{4 q_2}(tL_n^2)\\&\ll q_2^{\frac{1}{2}}.
\end{align*}
If now $\alpha\geq 3$, we still have the implication $q_1q_2|L_n^2$, because $q_2$ is squarefree. In that case the same argument as for $(q_1,2)=1$ can be applied. In the other cases we have $2^\gamma|| q$ for some $\gamma\leq 5$. Those finitely many cases can be dealt with via the multiplicativity of $|Q((\bm{m},0),q)|$, only possibly increasing the implied constants.

We arrive at
\begin{align*}
Q((\bmh{m},0),q)&\ll q_1^{n+2}q_2^{\frac{3}{2}}{\sum_{t(q_1)}}^*\bigl|\sum_{\substack{\bmh{L}(q_1q_2)\\ \substack{q_1|t\nabla C(\bmh{L})+\bmh{m}\\q_1q_2|C(\bmh{L})}}}e_q(tC(\bmh{L})+\bmh{m}.\bmh{L})\bigr|.
\end{align*}
This double sum is the same object as Heath-Brown considers in section 7 of \cite{hb2}. For $n\geq 5$ we can apply his argument to get
\begin{align*}
\sum_{\substack{||\bmh{m}||\leq y\\F_1(\bmh{m})=0}}|Q((\bmh{m},0),q)|\ll q^\epsilon q_1^{2n+2}q_2^{\frac{1}{2}n+2}\max_{\bmh{r}(q_1)}N_1(y,\bmh{r},q_1),
\end{align*}
and
\begin{align*}
\sum_{\substack{||\bmh{m}||\leq y\\\substack{F_1(\bmh{m})=0\\ F_2(\bmh{m})=0}}}|Q((\bmh{m},0),q)|\ll q^\epsilon q_1^{2n+2}q_2^{\frac{1}{2}n+2}\max_{\bmh{r}(q_1)}N_2(y,\bmh{r},q_1),
\end{align*}
where $N_1$ and $N_2$ are es defined in \eqref{N1def} and \eqref{N2def}. We apply the first statement of Lemma \ref{Nlemma} and change the notation from $q=q_1^2q_2$ to $q=k_3l_4$ to get the estimate
\begin{align*}
\sum_{\substack{||\bmh{m}||\leq y\\F_1(\bmh{m})=0}}|Q((\bmh{m},0),q)|&\ll q^\epsilon q_1^{2n+2}q_2^{\frac{1}{2}n+2}\bigl(y/q_2+1 \bigr)^{n-1}\\
&\ll q^\epsilon\bigl( y^{n-1} q_1^{n+3} q_2^{\frac{1}{2}n+2}+q_1^{2n+2}q_2^{\frac{1}{2}n+2}\bigr)\\
&\ll (k_3l_4)^\epsilon(k_3^{\frac{5}{6}n+\frac{4}{3}}l_4^{n+1}+y^{n-1}(k_3l_4)^{\frac{1}{2}n+\frac{5}{3}}).
\end{align*}
Similarly, now applying the second statement of Lemma \ref{Nlemma} we deduce
\begin{align*}
\sum_{\substack{||\bmh{m}||\leq y\\\substack{F_1(\bmh{m})=0\\ F_2(\bmh{m})=0}}}|Q((\bmh{m},0),q)|&\ll q^{\epsilon}q_1^{2n+2}q_2^{\frac{1}{2}n+2}\bigl(y/q_1+1 \bigr)^{n-2}\\
&\ll q^\epsilon \bigl(y^{n-2}q_1^{n+4}q_2^{\frac{1}{2}n+2}+q_1^{2n+2}q_2^{\frac{1}{2}n+2} \bigr)\\
&\ll (k_3l_4)^\epsilon(k_3^{\frac{5}{6}n+\frac{4}{3}}l_4^{n+1}+y^{n-2}(k_3 l_4)^{\frac{1}{2}n+2}).
\end{align*}
\end{proof}
\subsection{Another Type of Average}
\label{3sec3.3}
As in Hooley's work on cubic forms we require exponential sum estimates of a type that does not appear in Heath-Brown's paper. We now consider 
\begin{align*}
D(k,b_{n+1})&:=\sum_{\bmh{b}(k)}|Q(\bm{b},k)|\\
E(k,r)&:={\sum_{c(k)}}^*D(k,cr).
\end{align*}
To put the strength of the next Lemma into context, we note that by Lemma \ref{primecase} and an estimate of the number bad $\bmh{m}$ we have the bounds
\begin{align*}
D(p,0)&\ll p^{\frac{3}{2}n+1} \\
E(p,r)&\ll p^{\frac{3}{2}n+2}.
\end{align*}
We now show that similar estimates hold for prime powers and that in the prime case the implicit constant can be chosen as $1$.
 
\begin{lem}\label{DundE}
$D(k,0)$ and $E(k)$ are multiplicative and we have
\begin{align*}
D(p^{\alpha},0)&\ll \sqrt{\alpha} p^{(\frac{3}{2}n+1)\alpha} \\
D(p,0)&\leq p^{\frac{3}{2}n+1}\\
E(p^{\alpha},r)&\ll \sqrt{\alpha} p^{(\frac{3}{2}n+2)\alpha}\\
E(p,r)&\leq  p^{\frac{3}{2}n+2}+O(p^{\frac{3}{2}n+\frac{3}{2}})
\end{align*}
\end{lem}
In this result the factor $\sqrt{\alpha}$ should be removable, but it simplifies the proof without preventing the intended application. 
\begin{proof}The multiplicativity follows directly from the pseudo-multiplicativity of $Q(\bm{b},k)$, see \eqref{pseudomult}. As in \cite{h9} the to estimate $D$ and $E$, we consider a second moment and then employ using the Cauchy-Schwarz inequality. We define
\begin{align*}
D_2(k,b_{n+1})&:=\sum_{\bmh{b}(k)}|Q(\bm{b},k)|^2\\
E_2(k,r)&:={\sum_{c(k)}}^*D_2(k,cr).
\end{align*}

We start the proof of the estimates by showing
\begin{align}\label{D_1b_n}
D(p^\alpha,b_{n+1})\ll \sqrt{\alpha} p^{\frac{1}{2}(3n+2)\alpha}
\end{align}
for arbitrary $b_{n+1}$ and $\alpha\geq 2$. This estimate is closely related to the similar result in section 11 of \cite{h9}. We follow the proof there and have 
\begin{align*}
D_2(p^\alpha,b_{n+1})=p^{n\alpha}\sum_{\substack{\substack{\bmh{l} (p^\alpha)\\ l_{n+1}(p^\alpha) } \\l_{n+1}' (p^\alpha)}}c_{p^\alpha}(C(\bmh{l})-l_{n+1}^2)c_{p^\alpha}(C(\bmh{l})-l_{n+1}'^2)e_{p^\alpha}(b_{n+1}(l_{n+1}-l_{n+1}'))
\end{align*}
with 
\begin{align*}
c_{p^\alpha}(d)=\begin{cases}p^{\alpha-1}(p-1) &\text{ if } p^\alpha|d, \\ -p^{\alpha-1} &\text{ if } p^{\alpha-1}||d,
\\ 0& \text{ otherwise.}
\end{cases}
\end{align*}
By setting
\begin{align*}
&R_1^\alpha(\bmh{l})=\#\{l_{n+1}(p^\alpha) : C(\bmh{l})-l_{n+1}^2\equiv 0 (p^\alpha)\},\\
&R_2^\alpha(\bmh{l})=\#\{l_{n+1}(p^\alpha) : C(\bmh{l})-l_{n+1}^2\equiv 0 (p^{\alpha-1}), C(\bmh{l})-l_n^2\not\equiv 0 (p^\alpha)\},
\end{align*}
we can can apply the triangle inequality and collect the terms to get
\begin{align*}
D_2(p^\alpha,b_{n+1})\leq p^{n\alpha}\sum_{\bmh{l}(p^\alpha)}[R_1^\alpha(\bmh{l})p^{\alpha-1}(p-1)+R_2^\alpha(\bmh{l})p^{\alpha-1}]^2.
\end{align*}
If $p^k||C(\bmh l)$, then
\begin{align*}
R_1^\alpha(\bmh l)\leq  2 p^{\frac{1}{2}k}
\end{align*}
and furthermore
\begin{align*}
R_2^\alpha(\bmh l)\leq p R_1^{\alpha-1}(\bmh l)\leq 2 p^{\frac{1}{2}k+1}.
\end{align*}
Therefore we get
\begin{align*}
D_2(p^\alpha,b_{n+1})&\ll p^{n\alpha}\sum_{k=0}^\alpha \sum_{\bmh l (p^\alpha): p^k||C(\bmh l)}p^{k+2\alpha}\\
&\ll p^{n\alpha}\sum_{k=0}^\alpha p^{k+2\alpha+k(n-1)+n(\alpha-k)}\\
&\ll \alpha p^{(2n+2)\alpha}.
\end{align*}
Applying the Cauchy-Schwarz inequality gives us 
\begin{align*}
D(p^\alpha,b_{n+1})\leq \Bigl(D_2(p^\alpha,b_{n+1})\sum_{\bmh{a}(p^\alpha)}1 \Bigr)^{\frac{1}{2}}\ll \sqrt{\alpha} p^{\frac{1}{2}(3n+1)\alpha}.
\end{align*}
This proves \eqref{D_1b_n}. The first estimate of the Lemma follows as the special case $b_{n+1}=0$. Further, the third estimate is deduced by summing \eqref{D_1b_n} over a suitable set of $b_{n+1}$

For the prime case, i.e. $\alpha=1$, we note thate for $p\neq 2$
\begin{align*} 
|Q((\bmh{m},0),p)|=p\Bigl|\sum_{\bmh{l}(p)}\legendre{C(\bmh{l})}{p}e_p(\bmh{m}.\bmh{l})\Bigr|.
\end{align*} 
We get

\begin{align*}
D_2(p,0)&=p^2\sum_{\bmh m (p)} \sum_{\bmh l, \bmh l' (p)} \legendre{C(\bmh l)}{p}\legendre{C(\bmh l')}{p}e_p(\bmh m(\bmh l - \bmh l'))\\
&= p^{(n+2)} \sum_{\bmh l} \legendre{C(\bmh l)}{p}^2\\
&\leq  p^{2n+2}.
\end{align*}
The result for $p=2$ is trivial and once more applying Cauchy-Schwarz shows
\begin{align*}
D(p,0)\leq p^{\frac{1}{2}(3n+2)},
\end{align*}
which was the second proposed estimate.

We now look at $E_2(p,r)$. If $p|r$ then the just proved result on $D_2(p,0)$ can be applied, in the other case we have for odd primes
\begin{align*}
E_2(p,r)&=p^{n+1}\sum_{\bm l (p)}c_p(C(\bmh l)-l_{n+1}^2)^2\\
&\leq p^{n+1}(\sum_{\bmh l(p)}\Bigl[\legendre{C(\bmh l)}{p}+1\Bigr](p-1)^2+p^{n+1})\\
&\leq p^{2n+3}+O(p^{2n+2}).
\end{align*}
The case $p=2$ again making no difficulty and one final application of Cauchy-Schwarz gives us 
\begin{align*}
E(p,r)\leq p^{\frac{1}{2}(3n+4)}+O(p^{\frac{1}{2}(3n+3)}),
\end{align*}
which completes the proof of Lemma \ref{DundE}.

\end{proof}

As in Hooley's treatment of the homogeneous case in \cite{h9} and \cite{h93}, we require slight improvements of Lemma \ref{DundE} in the prime case to show Theorem \ref{Thm1.1}. They were provided in \cite{h9} by the use of Katz' paper \cite{katz}.  We start with the easier case of $E(p,1)$, that follows directly from the Katz' work.
\begin{lem}\label{KfurE}
There is a set of primes $\mathcal{P}_E$ having positive Dirichlet-density and $C_1<1$ such that for all $w\in \mathcal{P}_E$ we have
\begin{align*}
E(w,1)< C_1 w^{\frac{1}{2}(3n+4)}.
\end{align*}
\end{lem}
\begin{proof}
We note that for $p\nmid \bm{b}$ we have
\begin{align*}
Q(\bm{b},p)=p\sum_{\substack{\bm{l}(p)\\ f(\bm{l})\equiv 0 (p)}}e_p(\bm{b}.\bm{l}).
\end{align*}
Thus 
\begin{align*}
E(p,1)=p\sum_{\bm{b}(p)}\bigl|\sum_{\substack{\bm{l}(p)\\ f(\bm{l})\equiv 0 (p)}}e_p(\bm{b}.\bm{l}) \bigr|-D(p,0).
\end{align*}
The sum is of the type considered in \cite{katz}. As $f(\bm{l})=C(\bmh{l})-l_{n+1}^2$ is a polynomial that can be written as a sum of homogeneous polynomials in disjoint variables and the leading one being of at least degree $3$, we have $A\geq 2$. Here $A$ is the $A$-number of \cite{katz}. For this case Katz proves the existence of a suitable set of primes $\mathcal{P}_E$ and $C_1'<1$, such that for all $w\in \mathcal{P}_E$
\begin{align*}
\sum_{\bm{b}(w)}\bigl|\sum_{\substack{\bm{l}(w)\\ f(\bm{l})\equiv 0 (w)}}e_w(\bm{b}.\bm{l}) \bigr|\leq C_1' w^{\frac{3}{2}n+1}.
\end{align*}
The Lemma follows by taking $C_1'<C_1<1$ and using the bound $D(p,0)\leq p^{\frac{3}{2}n+1}$.
\end{proof}

The following similar estimate of $D(p,0)$ requires more delicate application of Katz' ideas. Its proof uses some results about sheaves stated in subsection \ref{3secapp}.

\begin{lem}\label{KfurD}
There is a set of primes $\mathcal{P}_D$ having positive Dirichlet-density and $C_2<1$ such that for all $w\in \mathcal{P}_D$ we have
\begin{align*}
D(w,0)< C_2 w^{\frac{1}{2}(3n+2)}.
\end{align*}
\end{lem}
\begin{proof} Let $A = 2^{n-1}$ and let $C_2  =1-\frac{1}{ 4 (1+A^2)}= 1 - \frac{1}{4(1 + 2^{ 2(n-1)})}$. Let $p=w$. We have $D(w,0)= \sum_{\bmh{b}(w)}|Q(\bm{b},w)|= \sum_{\bmh{m} \in \mathbb F_p^{n-1}} p \Bigl|\sum_{\bmh{l}(p)}\legendre{C(\bmh{l})}{p}e_p(\bmh{m}.\bmh{l})\Bigr| $ so the stated bound is equivalent to

\[   \sum_{\bmh{m} \in \mathbb F_p^{n-1}} \left| \sum_{\bmh{l} \in \mathbb F_p^{n}} \left( \frac{ C(\bmh{l} ) }{ p} \right) e_p ( \bmh{m} \cdot \bmh{l}) \right| <  \left(1 - \frac{1}{4(1+A^2)}\right) p^{\frac{3}{2}n} .\]

That this bound holds for a positive proportion of $p$ will be a special case of \cite[Theorem 4.10]{katz}. We will first explain Katz's notation, and how we specialize it, as well as his hypotheses, and why they hold in our case.

Katz works with a ring $R$, an integer $r\geq 1$, a rank $r$ vector bundle $E$ on $\operatorname{Spec} R$, a closed subscheme $X$ of $E$, $i: X\to E$ the inclusion, an open set $V$ of $X$,  $j: V \to X$ the inclusion, and an integer $n$.  As below in \S3.4, we will take $R = \mathbb Z[1/2]$, $r=n$, $E$ the trivial vector bundle of rank $n$, $X=E$, $V$ the subset where $C \neq 0$, and Katz's $n$ is our $n$. About these Katz assumes that $R \subseteq \mathbb C$ is smooth and finitely generated over $\mathbb Z$, which is obvious, as well as (in \cite[Hypotheses 4.0.1, 4.0.2, and 4.0.3]{katz}) that $V$ is smooth and that the geometric fibers of $X$ and $V$ are irreducible of dimension $n$, which is immediate in our case.

Next in \cite[(4.2)]{katz} Katz takes a constructible $\overline{\mathbb Q}_\ell$-sheaf $\mathcal F$ on $X[1/\ell]$ which is mixed of weight $\leq 0$, adapted to some stratification of $X$, and such that $j^* \mathcal F$ is lisse and pure of weight zero. We take for $\mathcal F$ the sheaf $\mathcal L_\chi(C(x_1,\dots,x_n))$ on $\mathbb A^n_{\mathbb Z}$, and these hypothesis are verified by Lemma \ref{lisse-pure}. Katz defines $K = j_{!*} ( j^* \mathcal F[n])$.

Finally, in the statement of \cite[Theorem 4.10]{katz}, Katz takes a constant $M''$, and we let $M''=0$. He assumes in addition that $\mathcal F$ is geometrically irreducible and the $A$-number of 
$\mathcal F$ is at least $2$. The geometric irreducibility follows from Lemma \ref{lisse-pure} and the $A$-number is at least $2$ by Lemma \ref{a-number}.

Katz's notation $|(i^* \mathcal F, F(s), s , \psi) | $  from (2.1), expands, in our case, where $i^* \mathcal F=\mathcal F$, $s=p$ and $F(s) = \mathbb F_p$, to  

\[ p^{-n} \sum_{\bmh{m} \in \mathbb F_p^{n}} \left| \sum_{\bmh{l} \in \mathbb F_p^{n}} \operatorname{trace}(\operatorname{Frop}_{\bmh{l},p} | \mathcal F_{\bmh{l}}) e_p ( \bmh{m} \cdot \bmh{l}) \right|\]
and by construction $\operatorname{trace}(\operatorname{Frop}_{\bmh{l},p} |\mathcal F_{\bmh{l}})=  \left( \frac{ C(\bmh{l} ) }{ p} \right) $ so Katz's bound \[p^{-n/2} |(i^* \mathcal F, F(s), s , \psi) |   \leq\left(1 - \frac{1}{4(1+A^2)}\right) \] specializes to our stated bound.

\end{proof}

\subsection{Lemmas on Sheaves}\label{3secapp}
We prove here some lemmas about $\ell$-adic sheaves that are used in the proofs of Lemmas \ref{primecase} and \ref{KfurD}. We will match as closely as possible the notation of \cite[section 4]{katz} in order to use some results from that paper.

We take $R = \mathbb Z[1/2]$, $r=n$, $E = \mathbb A^{n}_{\mathbb Z[1/2]}$ the trivial rank ${n}$ vector bundle, $E^\vee$ the trivial dual bundle. We let $X= \mathbb A^{n}_{\mathbb Z[1/2]}$ be all of $E$, and let $V$ be the subset of $X$ where $C \neq 0$. 

We let $\ell$ be any prime other than two, and let $\mathcal F$ be the sheaf $\mathcal L_\chi (C(x_1,\dots,x_{n}))$ on $\mathbb A^{n}_{\mathbb Z[1/2\ell]}$, where $\chi$ is the quadratic character. (As usual, the Kummer sheaf of the polynomial is defined to vanish on the vanishing locus of $C$.)

\begin{lem}\label{lisse-pure} The restriction of $\mathcal F$ to $V[1/\ell]$ is lisse, pure of weight zero, and geometrically irreducible when restricted to $V_{\mathbb F_p}$ for each prime $p \neq 2,\ell$.  \end{lem}

\begin{proof} $\mathcal F$ is the extension by zero from $V[1/\ell]$ of the lisse sheaf defined by the unique nontrivial one-dimensional representation of the fundamental group of $V[1/\ell]$ that factors through the automorphism group of the finite etale double cover $y^2 = C(x_1,\dots,x_{n})$. Hence it is lisse, and because its monodromy group has order two, each Frobenius element has order $1$ or $2$, and thus is pure of weight zero. It is geometrically irreducible because it corresponds to a one-dimensional representation, and all such representations are irreducible. \end{proof}

Let $j: V[1/\ell] \to X[1/\ell]$ be the inclusion. Following Katz, let  $K= j_{!*} (j^* \mathcal F[{n}]) $, where $j_{!*}$ is the middle extension of perverse sheaves.

The following simple lemma about \'{e}tale cohomology has probably appeared in the literature before:

\begin{lem}\label{easy-etale} Let $Y$ be a smooth scheme over $\mathbb Z[1/\ell]$, let $D$ be a strict normal crossings divisor on $Y$ relative to $\mathbb Z$, let $j: Y- D \to Y$ be the open inclusion, and let $\mathcal L$ be a lisse $\mathbb Q_\ell$-sheaf of rank one on $Y-D$, with nontrivial local monodromy around the generic point of each irreducible component of $D$.  Then $Rj_* \mathcal L$ vanishes on $D$. \end{lem}

\begin{proof} This is an \'{e}tale-local statement, so we may work \'{e}tale-locally, and thus we may assume $Y = \mathbb A^n$ and $D$ is given by the equation $x_1 x_2 \dots x_m=0$ for some $m \leq n$. By Abyankhar's lemma, in this case $\mathcal L = \otimes_{i=1}^m \mathcal L_{\chi_i} (x_i)$. We may apply the K\"{u}nneth formula to calculate the pushforward, which reduces us to the case where $n=1$, where we must check the stalk of the pushforward at $0$ vanishes. But that stalk is a complex consisting of the inertia invariants and coinvariants of the local monodromy representation, which both vanish as it is one-dimensional and nontrivial. \end{proof}

\begin{lem}\label{middle-extension} We have an isomorphism $K= \mathcal F[{n}]$. In particular, $\mathcal F[n]$ is perverse. \end{lem}

\begin{proof}  The open set $V$ of $\mathbb A^{n}$ where $C \neq 0$ is also an open subset of the blowup of $\mathbb A^{n}$ at the origin. Let $j'$ be the inclusion of this open set. Then we first show that $R j'_* j^* \mathcal F[n] $ vanishes away from the image of $j'$. To do this, we observe that $j'$ is the complement of the normal crossings divisor $D$ containing two components, the strict transform of the vanishing locus of $C$ and the exceptional divisor. Because $C$ is irreducible, it vanishes to order $1$ on generic point of the strict transform of its vanishing locus, and because it is homogeneous of degree $3$, it vanishes to order $3$ at the generic point of the exceptional divisor. Because both these numbers are odd, $\mathcal L_\chi(C)$ has nontrivial local monodromy around the vanishing sets of these divisors. Hence by Lemma \ref{easy-etale}, the natural map $ R j'_* j^* \mathcal F[n] $ vanishes away from the image of $j'$. 

By the Leray spectral sequence, the pushforward of $ R j'_* j^* \mathcal F[n] $ from the blow-up of $\mathbb A^{n}$ to $\mathbb A^{n}$ is $R j_* j^* \mathcal F[n]$. By the proper base change formula, the stalk of $ R j'_* j^* \mathcal F[n] $ vanishes away from the image of $j$. Hence the natural map $R j_! j^* \mathcal F[n] \to R j_* j^* \mathcal F[n]$ is an isomorphism. Because $Rj_{!*}j^* \mathcal F[n]$ is defined as the image of that map on zeroth perverse homology, it is isomorphic to both sides as well.

Because $\mathcal F[n]$ is the middle extension of a perverse sheaf, it is perverse.\end{proof}

In \cite[Lemma 4.6]{katz}, Katz defines a notion of the $A$-number of a sheaf $\mathcal F$ and proves it is equivalent to the notion of $A$-number of a perverse sheaf defined in \cite[p. 127]{fok} as the rank of the Fourier transform of the associated middle extension sheaf $K$. In the next lemma, we calculate this $A$-number, using the formula of \cite[Lemma 4.6]{katz}. (It is important for us that these two definitions give the same number as we use this lemma once in the setting of \cite{katz} and once in the setting of \cite{fok}.)
\begin{lem}\label{a-number} The $A$-number of $\mathcal F$ is $2^{n-1}$. \end{lem}

\begin{proof} Let $H$ be a general hyperplane in $\mathbb A^{n}_\mathbb C$, which is some translation of a general hyperplane $H'$ through the origin. Let us first calculate the Euler characteristic of the vanishing set in $H$ of $C$. Every line through the origin either intersects exactly one point of $H$ or lies in $H'$, but not both, so the  projection from zero to infinity, defines an isomorphism between $H$ and $\mathbb P^{n-1} - \mathbb P(H')$, so this vanishing set is equal to the complement inside the vanishing set in $\mathbb P^{n-1}$ of a nonsingular cubic form of the vanishing set in $\mathbb P^{n-2}$ of a nonsingular cubic form, hence its Euler characteristic is the difference of the Euler characteristics of two smooth cubic hypersurfaces.   These have Euler characteristics $(1/3)  ( (-2)^{n} - 1) + n-1$ and $(1/3) ( (-2)^{n-1} -1 ) + n-2$ respectively, so the vanishing set of $C$ in $H$ has Euler characteristic $- (-2)^{n-1}+1= 1 - (-2)^{n-1} $.

The complement of the vanishing set of $H$ in $C$ is $V_C \cap H$. Because $H$, a hyperplane, has Euler characteristic $1$, it has Euler characteristic $(-2)^{n-1}$.

By \cite[Lemma 4.6]{katz}, the $A$-number of $\mathcal F$ is the Euler characteristic of $K$ (as a sheaf on $\mathbb A^{n}_{\mathbb C})$ minus the Euler characteristic of the restriction of $K$ to a general hyperplane $H$. By Lemma \ref{middle-extension}, this is $(-1)^{n}$ times $\chi(\mathbb A^{n}_{\mathbb C}, \mathcal F) - \chi(H, \mathcal  F)$. Because the Euler characteristic of the extension of a sheaf by zero is the Euler characteristic of the original sheaf, this is $\chi(V_{\mathbb C}, \mathcal L_\chi ( C(x_1,\dots, x_{n}))) - \chi ( V_{\mathbb C} \cap H, \mathcal L_\chi ( C(x_1,\dots, x_{n})))$. Because $\mathcal L_\chi$ is locally trivial of rank one and the Euler characteristic is local, this is the same as $\chi (V_{\mathbb C}) - \chi(V_{\mathbb C} \cap H)$. Because $V_{\mathbb C}$ is the complement of the solution set of a homogeneous polynomial equation, it admits a free action of $S^1$ by multiplying each coordinate by unit complex numbers, and hence has Euler characteristic zero.  We have seen that $\chi(V_{\mathbb C} \cap  H)$ has Euler characteristic $(-2)^{n-1}$, so \[A(\mathcal F) = (-1)^{n} ( 0 - (-2)^{n-1} ) = 2^{n-1}.\]

\end{proof}

\section{Additional Results} 
\label{additional results}
We have gathered all the necessary estimates related to the exponential sums and the appearing integrals. To put these together we need some results on objects that will appear in the final estimation. We start by restating Lemma 22 and 23 of \cite{h9}.
\begin{lem}\label{HoL22}
If $\Delta ,l\neq0$ and $A_{2}\geq 1$, then
\begin{align*}
\sum_{\substack{k\leq y\\(k,l)=1}}(k,\Delta)^{\frac{1}{2}}A_{2}^{\omega(k)}\ll \sigma_{-\frac{1}{4}}(\Delta,l)y(\log 2y)^{A_{2}-1},
\end{align*}
where
\begin{align*}
\sigma_{s}(\Delta,l)=\sum_{\substack{d|\Delta\\(d,l)=1}}d^{s}.
\end{align*}
\end{lem}

\begin{lem}\label{HoL23}
Let $A_2\geq 1$, $A_3>0$, \begin{align*}
Z=X^{A_{3}/\log\log X},
\end{align*}
and $k_1^\dag$ be a squarefree number composed only of prime numbers less than $Z$. There exists and $A_4$ such that for
\begin{align*}
X< u \leq X^2 \text{,  } \delta\neq 0\text{,  }l\leq X^{\frac{1}{8}}\text{,  }(c,l)=1,
\end{align*}
we have
\begin{align*}
\sum_{\substack{k_1^\dag\leq u\\k\equiv c(l)}}A_{2}^{\omega(k_1^\dag)}(\Delta,k_1^\dag)^{\frac{1}{2}}\ll \frac{u(\log \log X)^{A_{4}}\sigma_{-\frac{1}{4}}(\Delta,l)}{\phi(l)\log X}
\end{align*}
and
\begin{align*}
\sum_{\substack{k_1^\dag\leq u\\(k_1^\dag,l)=1}}A_{2}^{\omega(k_1^\dag)}(\Delta,k_1^\dag)^{\frac{1}{2}}\ll \frac{u(\log \log X)^{A_{4}}\sigma_{-\frac{1}{4}}(\Delta,l)}{\log X}.
\end{align*}
\end{lem}
The next results deal with sums in which both $Q$ and $J$ appear. The three appearing cases are related to the different average results for $Q(\bm m,y)$ as stated in Lemma \ref{6.1} and Lemma \ref{6.2}. 
\begin{lem} \label{Jsum}
For $l_3\leq A_1X^{\frac{3}{2}}$. We then have for $(a,l_3)=1$ and any $M\geq 1$
\begin{align*}
\sum_{\bmh{m}\neq 0}||\bm{m}||_X^{\epsilon}|Q(a\bm{m},l_3)|J(\bm{m},Y)\ll_M& \Bigl(\frac{1}{|m_{n+1}|+1}\Bigr)^M \Bigl\{ \frac{X^{\epsilon}\min(X,Y)^{\frac{1}{2}n-\frac{1}{2}}l_3^{\frac{5}{6}n+\frac{3}{2}}}{X^{\frac{1}{2}n-\frac{1}{2}}}\\&+\frac{X^{\epsilon}Y^{\frac{1}{2}n-\frac{1}{2}}N^{\frac{1}{2}n+\frac{1}{2}}l_3^{\frac{1}{2}n+\frac{3}{2}}}{X^{n}}\Bigr \}.
\end{align*}
\end{lem}
\begin{proof}
The result is the analogue of Lemma 15 of \cite{h9}. Consider first the case of $m_n=0$. We combine \eqref{Jdef} and Lemma \ref{6.1} by dyadic dissection with the condition $2^\alpha U\leq ||\bmh{m}|| < 2^{\alpha+1}U$, where $U=1$ if $Y\leq X$ but $U=Y/X$ if $Y>X$ to prove the Lemma. 
For $m_n\neq 0$ do the same with $M=2M'$ and use the inequality
\begin{align}\label{m_nloesung}
\Bigl(\frac{N}{X\lVert \bm{m} \rVert_X} \Bigr)^{2{M'}} \leq \Bigl(\frac{2N}{X\lVert \bmh{m} \rVert+|X^{\frac{1}{2}}m_{n+1}|} \Bigr)^{2{M'}}\leq \Bigl(\frac{N}{X\lVert\bmh{m} \rVert} \Bigr)^{M'} \Bigl(\frac{1}{|m_{n+1}|+1} \Bigr)^{M'}.
\end{align}
\end{proof}

\begin{lem}\label{JsumD0}
For $l_3\leq A_1X^{\frac{3}{2}}$ we have 
\begin{align*}
\sum_{\substack{\bm{m}\neq 0\\ \substack{F_1(\bmh{m})=0\\ F_2(\bmh{m})=0}}}||\bmh{m}||_X^{\epsilon}|Q((\bmh{m},0),k_3l_4)|J((\bmh{m},0),Y)\ll& \frac{X^{\epsilon}\min(X,Y)^{\frac{1}{2}n-\frac{1}{2}}k_3^{\frac{5}{6}n+\frac{4}{3}}l_4^{n+1}}{X^{\frac{1}{2}n-\frac{1}{2}}}\\&+\frac{X^{\epsilon}Y^{\frac{1}{2}n-\frac{1}{2}}N^{\frac{1}{2}n-\frac{3}{2}}l_3^{\frac{1}{2}n+2}}{X^{n-2}}
\end{align*}
and 
\begin{align*}
\sum_{\substack{\bm{m}\neq 0\\ F_1(\bmh{m})=0}}||\bmh{m}||_X^{\epsilon}|Q((\bmh{m},0),k_3l_4)|J((\bmh{m},0),Y)\ll& \frac{X^{\epsilon}\min(X,Y)^{\frac{1}{2}n-\frac{1}{2}}k_3^{\frac{5}{6}n+\frac{4}{3}}l_4^{n+1}}{X^{\frac{1}{2}n-\frac{1}{2}}}\\&+\frac{X^{\epsilon}Y^{\frac{1}{2}n-\frac{1}{2}}N^{\frac{1}{2}n-\frac{1}{2}}l_3^{\frac{1}{2}n+\frac{5}{3}}}{X^{n-1}}.
\end{align*}
\end{lem}
\begin{proof}

The result is obtained in the same way as the Lemma before, except now using Lemma \ref{6.2}.

\end{proof}

The next type of result we need is related to the estimates of $D$ and $E$. To use them we have sort the sums over $\bmh{m}$ into residue classes and so will need a result that estimates sums over $J$ of this type, which in \cite{h9} are Lemma 24 and 25.

\begin{lem} \label{Jkongruentsumme} For $Y>X^{\frac{5}{4}}$,$l\leq X^{\frac{1}{7}}$, we have for all $M\geq 1$
\begin{align*}
\sum_{\bm{\hat{m}}\equiv \bm{\hat{b}} (l)}J((\bm{\hat{m}},m_{n+1}),Y)\ll_M \Bigl(\frac{1}{|m_{n+1}|+1}\Bigr)^M\frac{Y^{\frac{1}{2}n-\frac{1}{2}}N^{\frac{1}{2}n+\frac{1}{2}}}{X^{n}l^{n}}\log^{n+1}\Bigl(\frac{N}{Y}\Bigr)
\end{align*}
and with $\sigma_{-\frac{1}{4}}(\Delta,l)$ as before, we have further
\begin{align*}
\sum_{\substack{\bm{\hat{m}}\equiv \bm{\hat{b}}(l)\\F_1(\bmh{m})\neq 0}}J((\bm{\hat{m}},0),Y)\sigma_{-\frac{1}{4}}(F_1(\bmh{m}),l)\ll\frac{Y^{\frac{1}{2}n-\frac{1}{2}}N^{\frac{1}{2}n+\frac{1}{2}}}{X^{n}l^{n}}\log^{n+1}\Bigl(\frac{N}{Y}\Bigr).
\end{align*}
\end{lem}
The proof of the first statement in the case $m_n=0$ is done again by dyadic dissection and using Lemma \ref{Integral}, noting that $\min(X,Y)=X$. For $m_n\neq 0$ we note that the case $\lVert\bm{m}\rVert_X\leq Y/X$ of \eqref{Jdef} does not happen and that further we have 
\begin{align*}
&\log^{n+1}(2X \lVert \bm{m} \rVert_X/Y)(Y/X\lVert \bm{m} \rVert_X)^{\frac{1}{2}n-\frac{1}{2}}(N/X\lVert \bm{m} \rVert_X )^M\\&\leq \log^{n+1}(2X \lVert \bmh{m} \rVert_X/Y)(Y/X\lVert \bmh{m} \rVert_X)^{\frac{1}{2}n-\frac{1}{2}}(\frac{1}{|m_{n+1}|+1} )^M
\end{align*}
in the range $Y/X< \lVert\bmh{m}\rVert \leq N/X$. Together with \eqref{m_nloesung} this completes the proof of the first statement. For the second statement 
\begin{align*}
\sum_{\substack{\lVert \bmh{m} \rVert \leq y\\ \bm{\hat{m}}\equiv \bm{\hat{b}}(l);F_1(\bmh{m})\neq 0}}\sigma_{-\frac{1}{4}}(F_1(\bmh{m}),l)
\end{align*}
takes the role of
\begin{align*}
\sum_{\substack{\lVert \bmh{m} \rVert \leq y \\\bm{\hat{m}}\equiv \bm{\hat{b}}(l)} }1.
\end{align*}
However, this does not change the argument, as shown in the proof of Lemma 25 of \cite{h9}.

\section{Proof of Theorem \getrefnumber{Thm1.1}}
\label{3sec5}
We are now in the position to begin the final estimation. Our goal is to show that $\Upsilon_1(X)$ dominates the other terms in \eqref{Updef}. For all $Y\leq A_1N$ the contribution of the sum over $\frac{1}{2}Y<k\leq Y$ to $\Upsilon_i(X)$ with $2\leq i\leq 5$ is by Lemma \ref{Integral} bounded by
\begin{align*}
&\sum_{\frac{1}{2}Y<k\leq Y}\sum_{\bm{m}}k^{-n-1}Q(\bm{m},k)I(\bm{m},k)\\
\ll&  Y^{-n-1}\sum_{\frac{1}{2}Y<k\leq Y}\sum_{\bm{m}}|Q(\bm{m},k)|J(\bm{m},Y)=P_i(Y),
\end{align*} 
say. Here the conditions for the sum over $\bm{m}$ are defined by the conditions on the analogue sum of $\Upsilon_i(X)$. Again with the sum over $\bm{m}$ having the appropriate condition given by \eqref{UPdef} we have then
\begin{align*}
\Upsilon_i(X)=\sum_{k\leq A_1N}\sum_{\bm{m}}k^{-n-1}Q(\bm{m},k)I(\bm{m},k)\ll\sum_{0\leq j\leq N_1}P_i(Y_j)
\end{align*}
with $N_1=[\log A_1N/ \log 2]$. We now consider seperately the different cases and proceed bounding $P_i(Y)$ in the next subsections. They are closely related in structure and argument to Hooley's approach as found in sections 15 to 18 of \cite{h9} and section 48 of \cite{h93}. 

We start our estimation by consider $\Upsilon_2$ and $\Upsilon_3$ in the following two subsections. For $\Upsilon_2$ the summation ranges over those $\bm{m}$ for which $m_{n+1}=0$, but $F(\bm{m})\neq 0$. In $\Upsilon_2$ the condition is $m_{n+1}\neq 0$. In both cases, for most $p$, the strongest estimates for $Q(\bm{m},p)$ and $Q(\bm{m},p^2$ provided by lemma \ref{primecase} and \ref{primesquare} are applicable for most $p$. However, this alone is not sufficient and it is also here where the estimates for $D$ and $E$ of both the simple kind (see lemma \ref{DundE}) and the elaborate kind (see lemmas \ref{KfurE} and \ref{KfurE}) play crucial role. Naturally, $D$ appears in the estimation of $\Upsilon_2$ and $E$ in the one of $\Upsilon_3$. It are these cases and the use of lemma \ref{KfurE} and \ref{KfurE} that prevents a saving better than $(\log X)^{-\delta}$ in Theorem \ref{goal}.

Afterwards, in subsections \ref{secup4} and \ref{secup5} we consider respectively $\Upsilon_4$ and $\Upsilon_5$. We are then in the bad cases respectively either $Q(\bm{m},p)$ and $Q(\bm{m},p^2)$ or only $Q(\bm{m},p^2)$. However, these worse estimates are compensated by the strength of Lemma \ref{JsumD0} that is based on Lemma \ref{6.2}. In these both cases $E$ and $D$ play no role and we obtain a power saving in a slightly technical but straight forward manner.

\subsection{$\Upsilon_2$}
We start with the estimation of $P_2(Y)$ and recall
\begin{align*}
P_2(Y)=Y^{-n-1}\sum_{\substack{\frac{1}{2}Y<k\leq Y \\  F_1(\bmh{m})\neq 0}}|Q((\bmh{m},0),k)|J((\bmh{m},0),Y).
\end{align*}
We now bound $P_2(Y)$ for $Y>X^{5/4}$. We introduce for some $A_5>0$ the auxiliary variable 
\begin{align*}
Z=X^{\frac{1}{16}A_{5} \log \log X}
\end{align*}
and for $k=k_1l_2$ define
\begin{align*}
k_1^*&=\prod_{\substack{p|k_1\\p\leq Z}}p\\
k_1^\dag&=\prod_{\substack{p|k_1\\p> Z}}p,
\end{align*}
so that $k_1^*k_1^\dag=k_1$. Then, as $k_1^*> X^{\frac{1}{16}}$ implies $\omega(k_1)\geq \omega(k_1^*)>A_{5}\log \log X$, we can rewrite $P_2(Y)$ as
\begin{align*}
P_2(Y)\leq Y^{-n-1}\bigl(\sum_{k_1^*,l_2\leq X^{\frac{1}{16}}}+\sum_{\substack{l_2\leq X^{1/16}\\\omega(k_1)>A_{5}\log \log X}}+\sum_{l_2\geq X^{\frac{1}{16}}}\bigr)={\sum}_1+{\sum}_2+{\sum}_3,
\end{align*}
say. 

We start by using the pseudomultiplicativity of $Q$ given by \eqref{pseudomult} to split up $\sum_1$ as follows. We write $k=k_1^*k_1^\dagger l_2$ and have
\begin{align*}
{\sum}_1=&Y^{-n-1}\sum_{\substack{k_1^*,l_2\leq X^{\frac{1}{16}}\\Y/2k_1^*l_2<k_1^\dagger\leq Y/k_1^*l_2}}\sum_{F_1(\bmh{m})\neq 0}|Q((\bmh{m},0),k)|J((\bmh{m},0),Y)\\
\ll& Y^{-n-1}\sum_{\substack{k_1^*l_2\leq X^{\frac{1}{8}}\\\substack{F_1(\bmh{m})\neq 0\\Y/2k_1^*l_2<k_1^\dagger\leq Y/k_1^*l_2}}}J((\bmh{m},0),Y)|Q(\overline{k_1^*l_2}(\bmh{m},0),k_1^\dagger)||Q(\overline{k_1^\dagger}(\bmh{m},0),k_1^*l_2)|.
\end{align*}
This gives us by applying Lemma \ref{primecase} and \eqref{betragmult} respectively on the two appearances of $Q$
\begin{align*}
\ll& Y^{-n-1}\sum_{\substack{k_1^*l_2\leq X^{\frac{1}{8}}\\ \substack{F_1(\bmh{m})\neq 0\\Y/2k_1^*l_2<k_1^\dagger\leq Y/k_1^*l_2}}} J((\bmh{m},0),Y)|Q((\bmh{m},0),k_1^*l_2)|A_{2}^{\omega(k_1^\dagger)}(k_1^\dagger,F_1(\bmh{m}))^{\frac{1}{2}}{k_1^\dagger}^{\frac{1}{2}n+1}\\
\ll& Y^{-\frac{1}{2}n}\sum_{\substack{k_1^*l_2\leq X^{\frac{1}{8}}\\ \substack{ F_1(\bmh{m})\neq 0\\ k_1^\dagger\leq Y/k_1^*l_2}}}\frac{1}{(k_1^*l_2)^{\frac{1}{2}n+1}}J((\bmh{m},0),Y)|Q((\bmh{m},0),k_1^*l_2)|A_{2}^{\omega(k_1^\dagger)}(k_1^\dagger,F_1(\bmh{m}))^{\frac{1}{2}}.
\end{align*}
We use the second part of the auxiliary Lemma \ref{HoL23} with $l=k_1^*l_2$, getting
\begin{align}\label{FEeq1}
 \ll& \frac{(\log \log X)^{A_{4}}}{Y^{\frac{1}{2}n-1}\log X}\sum_{\substack{k_1^*l_2\leq X^{\frac{1}{8}}\\F_1(\bmh{m})\neq 0}}\frac{1}{(k_1^*l_2)^{\frac{1}{2}n+2}}J((\bmh{m},0),Y)|Q((\bmh{m},0),k_1^*l_2)|\sigma_{-\frac{1}{4}}(F_1(\bmh{m}),k_1^*l_2).
\end{align}

We now sort the values of $\bmh{m}$ in the innermost sum into congruence classes modulo $k_1^*l_2$. 
\begin{align*}
&\sum_{F_1(\bmh{m})\neq 0}J((\bmh{m},0),Y)|Q((\bmh{m},0),k_1^*l_2)|\sigma_{-\frac{1}{4}}(F_1(\bmh{m}),k_1^*l_2)\\&=\sum_{0<\bmh{b}<k_1^*l_2}|Q((\bmh{b},0),k_1^*l_2)|\sum_{\substack{F_1(\bmh{m})\neq 0\\ \bmh{m}\equiv \bmh{b} (k_1^*l_2)}}J((\bmh{m},0),Y)\sigma_{-\frac{1}{4}}(F_1(\bmh{m}),k_1^*l_2).
\end{align*}
Since we are considering currently $Y>X^{5/4}$ we can apply the second part of Lemma \ref{Jkongruentsumme} to estimate
\begin{align*}
\sum_{0<\bmh{b}<k_1^*l_2}|Q((\bmh{b},0),k_1^*l_2)|\sum_{\substack{F_1(\bmh{m})\neq 0\\ \bmh{m}\equiv \bmh{b} (k_1^*l_2)}}J((\bmh{m},0),Y)\sigma_{-\frac{1}{4}}(F_1(\bmh{m}),k_1^*l_2)\\
\ll \frac{Y^{\frac{1}{2}n-\frac{1}{2}}N^{\frac{1}{2}n+\frac{1}{2}}}{X^{n}}\log^{n}\Bigl(\frac{N}{Y}\Bigr) \frac{D(k_1^*l_2,0)}{(k_1^*l_2)^{n}}.
\end{align*}
Going back to \eqref{FEeq1} and relaxing $k_1^*l_2$ to $l$ we obtain the estimate
\begin{align*}
{\sum}_1\ll \frac{Y^{\frac{1}{2}}(\log \log X)^{A_{4}}N^{\frac{1}{2}n+\frac{1}{2}}}{X^{n}\log X}\log^{n+1}\Bigl(\frac{N}{Y}\Bigr)\sum_{l\leq X}\frac{D(l,0)}{l^{\frac{3}{2}n+2}}.
\end{align*}
Now, by using Lemma \ref{DundE}, \ref{KfurD}, and a Mertens formula on the sum over $l$, we get for some $\delta>0$ depending on $C_1$ and $\mathcal{P}_D$ 
\begin{align*}
\sum_{l\leq X}\frac{D(l,0)}{l^{\frac{3}{2}n+\frac{1}{2}}}\ll (\log X)^{1-2\delta}.
\end{align*} 
This implies
\begin{align}\label{sum1}
{\sum}_1 \ll \frac{Y^{\frac{1}{2}}N^{\frac{1}{2}n+\frac{1}{2}}}{X^{n}(\log X)^\delta}\log^{n+1}\Bigl(\frac{N}{Y}\Bigr).
\end{align}

For ${\sum}_2$ we start similarly, now with $l_2$ and $k_1$ taking the roles of $k_1^*l_2$ and $k_1^\dag$. The same steps as for ${\sum}_1$ then give us

\begin{align*}
&{\sum}_2 \ll \\ \ll& Y^{-\frac{1}{2}n}\sum_{\substack{l_2\leq X^{\frac{1}{16}}\\F_1(\bmh{m})\neq 0}}\frac{J((\bmh{m},0),Y)|Q((\bmh{m},0),l_2)|}{l_2^{\frac{1}{2}n+1}}\!\!\sum_{\substack{k_1\leq Y/l_2\\ \omega(k_1)> A_{5}\log \log X}}A_{2}^{\omega(k_1)}(k_1,F_1(\bmh{m})^{\frac{1}{2}}).
\end{align*}
For $A_{5}$ sufficiently large we use Lemma \ref{HoL22} and bound the innermost sum by
\begin{align*}
e^{-A_{5}\log \log X}\sum_{k_1\leq Y/l_2}(A_{2}e)^{\omega(k_1)}(k_1,F_1(\bmh{m}))^{\frac{1}{2}}&\ll \frac{Y}{l_2}\log ^{A_{2}e-1-A_{5}}(X)\sigma_{-\frac{1}{4}}(F_1(\bmh{m}),l_2)\\
&\ll \frac{Y\sigma_{-\frac{1}{4}}(F_1(\bmh{m}),l_2)}{l_2 \log X}.
\end{align*}
Using this and then applying the same steps as for ${\sum}_1$ gives us

\begin{align*}
{\sum}_2\ll \frac{Y^{\frac{1}{2}}N^{\frac{1}{2}n+\frac{1}{2}}}{X^{n}\log X}\log^{n+1}\Bigl(\frac{N}{Y} \Bigr)\sum_{l_2\leq X}\frac{D(l_2,0)}{l_2^{\frac{3}{2}n+2}}.
\end{align*}
By Lemma \ref{DundE} and using that the sum ranges only over square full numbers, we have
\begin{align*}
\sum_{l_2\leq X}\frac{D(l_2,0)}{l_2^{\frac{3}{2}n+2}}\ll 1,
\end{align*}
hence we estimate
\begin{align}\label{sum2}
{\sum}_2\ll \frac{Y^{\frac{1}{2}}N^{\frac{1}{2}n+\frac{1}{2}}}{X^{n}\log X}\log^{n+1}\Bigl(\frac{N}{Y} \Bigr).
\end{align}
The important saving of $\log X$ now stems from the summation condition on $\omega(k_1)$.

The following transformation for ${\sum}_3$ is also the starting point to handle those $P_2(Y)$ for which $Y\leq X^{5/4}$ as well well as for $P_3(Y)$ and $P_4(Y)$ for all $Y$. We write $k=k_1k_2l_3$ and by using the multiplicativity of $|Q|$ have

\begin{align}\label{Ho139}
&Y^{-n-1}\sum_{\substack{l_3\\\bmh{m}}}J((\bmh{m},0),Y)\sum_{k_1,k_2}|Q((\bmh{m},0),k_1k_2l_3)| \\ \nonumber
=&Y^{-n-1 }\sum_{\substack{l_3\\\bmh{m}}}|Q((\bmh{m},0),l_3)|J((\bmh{m},0),Y)\sum_{k_1,k_2}|Q((\bmh{m},0),k_1k_2)|,
\end{align}
where the restrictions on the sums are given by the $P_i(Y)$ they appear in.
In ${\sum}_3$ and for $P_2(Y)$ with $Y<X^{5/4}$ the condition is $F_1(\bmh{m})\neq 0$ and Lemma \ref{primecase} and \ref{primesquare} give us terms of the form

\begin{align*}
Y^{-n-1} \sum_{\substack{l_3\\F_1(\bmh{m})\neq 0}}|Q((\bmh{m},0),l_3)|J((\bmh{m},0),Y)\sum_{k_1,k_2}A_{2}^{\omega(k_1k_2)}(k_1k_2)^{\frac{1}{2}n+1}(F_1(\bmh{m})^2,k_1k_2))^{\frac{1}{2}}.
\end{align*}
For ${\sum}_3$ the appearing conditions on $k$ are $\frac{1}{2}Y<k_1k_2l_3\leq Y$ and $l_2=k_2l_3\geq X^{1/16}$. This can be relaxed to the disjunction of

\begin{align*}
l_3<X^{1/16}\text{, } X^{1/16}/l_3\leq k_2\leq Y/l_3 \text{, } k_1\leq Y/(k_2l_3)\\
X^{1/16}\leq l_3\leq Y\text{ , }k_1k_2\leq Y/l_3.
\end{align*}
In the first case the sum over $k_1k_2$ becomes
\begin{align*}
&O\Bigl(X^\epsilon\sum_{X^{1/16}/l_3\leq k_2\leq Y/l_3}k_2^{\frac{1}{2}n+1}((F_1(\bmh{m}))^2,k_2)^{\frac{1}{2}}\sum_{k_1\leq Y/(k_2l_3)}k_1^{\frac{1}{2}n+1}(F_1(\bmh{m}),k_1)^{\frac{1}{2}}\Bigr)\\
=&O\Bigl(\frac{||\bmh{m}||^\epsilon X^\epsilon Y^{\frac{1}{2}n+2}}{X^{1/32}l_3^{\frac{1}{2}n+\frac{3}{2}}} \Bigr),
\end{align*}
by using Hooleys calculations leading to \cite[(143)]{h9}. In the second it is
\begin{align*}
O\Bigl(X^\epsilon\sum_{l\leq Y/l_3}l^{\frac{1}{2}n+1}((F_1(\bmh{m}))^2,l)^{\frac{1}{2}} \Bigr)=O\Bigl(\frac{||\bmh{m}||^\epsilon X^\epsilon Y^{\frac{1}{2}n+2}}{l_3^{\frac{1}{2}n+2}} \Bigr).
\end{align*}
In the first case the contribution to ${\sum}_3$ is therefore, by using Lemma \ref{Jsum} with $m_n=0$ and noting that $\min(X,Y)=X$, as still $Y>X^{\frac{5}{4}}$,

\begin{align*}
&\frac{X^\epsilon}{X^{1/32}Y^{\frac{1}{2}n-1}}\sum_{l_3<X^{1/16}}\frac{1}{l_3^{\frac{1}{2}n+\frac{3}{2}}}\sum_{\bmh{m}\neq 0}||\bmh{m}||^\epsilon|Q((\bmh{m},0),l_3)|J((\bmh{m},0),Y)\\
&\ll \frac{X^\epsilon}{X^{1/32}Y^{\frac{1}{2}n-1}}\Bigl(\sum_{l_3<X^{1/16}}l_3^{\frac{1}{3}n}+\frac{Y^{\frac{1}{2}n-\frac{1}{2}}N^{\frac{1}{2}n+\frac{1}{2}}}{X^{n}}\Bigr)\\
&\ll \frac{X^{\frac{1}{48}n-\frac{1}{32}+\frac{1}{48}+\epsilon}}{Y^{\frac{1}{2}n-1}}+\frac{Y^{\frac{1}{2}}N^{\frac{1}{2}n+\frac{1}{2}}}{X^{n+\frac{1}{32}-\epsilon}}.
\end{align*}
In the latter case it is, again by applying Lemma \ref{Jsum},

\begin{align*}
&\frac{X^\epsilon}{Y^{\frac{1}{2}n-1}}\sum_{X^{1/16}\leq l_3 \leq Y}\frac{1}{l_3^{\frac{1}{2}n+2}}\sum_{\bmh{m}\neq 0}||\bmh{m}||^\epsilon|Q((\bmh{m},0),l_3)|J((\bmh{m},0),Y)\\
&\ll \frac{X^\epsilon}{Y^{\frac{1}{2}n-1}}\Bigl(\sum_{X^{1/16}\leq l_3 \leq Y}l_3^{\frac{1}{3}n+\frac{1}{2}}+\frac{Y^{\frac{1}{2}n-\frac{1}{2}}N^{\frac{1}{2}n+\frac{1}{2}}}{X^{n}l_3^{\frac{1}{2}}}\Bigr)\\
&\ll \frac{X^\epsilon}{Y^{\frac{1}{6}n-\frac{5}{6}}}+\frac{Y^{\frac{1}{2}}N^{\frac{1}{2}n+\frac{1}{2}}}{X^{n+\frac{1}{32}-\epsilon}}.
\end{align*}
Therefore we have
\begin{align}\label{sum3}
{\sum}_3\ll \frac{X^\epsilon}{Y^{\frac{1}{6}n-\frac{5}{6}}}+\frac{Y^{\frac{1}{2}}N^{\frac{1}{2}n+\frac{1}{2}}}{X^{n+\frac{1}{32}-\epsilon}}.
\end{align}
For $Y\leq X^{5/4}$ we lose the condition $X^{1/16}\leq l_3$ and get by similar calculations

\begin{align}\label{sum4}
P_2(Y)\ll \frac{X^\epsilon \min (X,Y)^{\frac{1}{2}n-\frac{1}{2}}}{X^{\frac{1}{2}n-\frac{1}{2}}Y^{\frac{1}{6}n-\frac{5}{6}}}+\frac{Y^{\frac{1}{2}}N^{\frac{1}{2}n+\frac{1}{2}}}{X^{n-\epsilon}}.
\end{align}
The contribution of to $\Upsilon_2(X)$ of $Y_j$ with $Y_j> X^{5/4}$ is, by putting together the bounds \eqref{sum1}, \eqref{sum2} and \eqref{sum3},
\begin{align*}
&O\Bigl(\frac{N^{\frac{1}{2}n+\frac{1}{2}}}{X^{n}(\log X)^\delta}\sum_{X^{5/4}<Y_j\leq A_1N}Y_j^{\frac{1}{2}}\log^{n+1} \frac{N}{Y_j} \Bigr)+O\Bigl(X^\epsilon \sum_{X^{5/4}<Y_j\leq A_1N} \frac{1}{Y_j^{\frac{1}{6}n-\frac{5}{6}}} \Bigr)\\&+O\Bigl(\frac{N^{\frac{1}{2}n}}{X^{n+\frac{1}{32}-\epsilon}}\sum_{X^{5/4}<Y_j\leq A_1N}Y_j^{\frac{1}{2}}\Bigr)
\end{align*}
Giving $n$ its intended value $6$ we get 
\begin{align*}
O\Bigl(\frac{X^{-\frac{1}{4}n+\frac{3}{2}}}{\log^\delta X}\Bigr)+O\Bigl(X^{-\frac{5}{24}(n-5)+\epsilon} \Bigr)+\Bigl(X^{-\frac{1}{4}n+\frac{3}{2}-\frac{1}{32}+\epsilon} \Bigr)\\
=O\Bigl(\frac{1}{\log^\delta X} \Bigr).
\end{align*}

For $Y_j\leq X^{5/4}$ we have by \eqref{sum4}
\begin{align*}
O\Bigl(X^\epsilon\sum_{X<Y_j\leq X^{5/4}}\frac{1}{Y_j^{\frac{1}{6}n-\frac{5}{6}}}\Bigr)+O\Bigl(\frac{X^\epsilon}{X^{\frac{1}{2}n-\frac{1}{2}}}\sum_{Y_j\leq X}Y_j^{\frac{1}{3}n+\frac{1}{3}} \Bigr)+O\Bigl(\frac{N^{\frac{1}{2}n+\frac{1}{2}}}{X^{n-\epsilon}}\sum_{Y_j\leq X^{5/4}}Y_j^{\frac{1}{2}} \Bigr),
\end{align*}
which is
\begin{align*}
&O\Bigl(X^{-\frac{1}{6}(n-5)+\epsilon} \Bigr)+O\Bigl(X^{-\frac{1}{6}(n-5)+\epsilon} \Bigr)+O\Bigl(X^{-\frac{1}{4}n+\frac{3}{2}-\frac{1}{8}+\epsilon} \Bigr)\\
&=O\Bigl(X^{-\frac{1}{8}+\epsilon} \Bigr).
\end{align*}
Overall we proved for $n\geq 6$ the existence of some $\delta>0$, such that
\begin{align*}
\Upsilon_2(X)\ll \frac{1}{\log^\delta X}.
\end{align*}

\subsection{$\Upsilon_3$}

To estimate $\Upsilon_3(X)$ we recall
\begin{align*}
P_3(Y)&=\sum_{m_{n+1}\neq 0}Y^{-n-1}\sum_{\substack{\frac{1}{2}Y<k\leq Y \\ \bmh{m}\neq 0}}|Q(\bm{m},k)|J(\bm{m},Y)\\
&=\sum_{m_{n+1}\neq 0}P_3(Y,m_{n+1}),
\end{align*}
say. We estimate $P_3(Y,m_{n+1})$ in a similar fashion as we did $P_2(Y)$. The condition $m_{n+1}$ takes the role of $F_1(\bmh{m})\neq 0$. This has some technical consequences that, however, do not fundamentally change the steps. We again start with the case $Y>X^{5/4}$ and as for $P_2(Y)$ dissect the sum over $k$ into 
\begin{align*}
P_3(Y,m_{n+1})\leq& Y^{-n-1}\bigl(\sum_{k_1^*,l_2\leq X^{\frac{1}{16}}}+\sum_{\substack{l_2\leq X^{1/16}\\\omega(k_1)>A_{5}\log \log X}}+\sum_{l_2\geq X^{\frac{1}{16}}}\bigr)\\
=&{\sum}_{4,m_{n+1}}+{\sum}_{5,m_{n+1}}+{\sum}_{6,m_{n+1}},
\end{align*}
say. 

Starting with the first of those, we follow the path of ${\sum}_1$ and get by using \eqref{pseudomult} and sorting the sums over $\bmh{m}$ and over $k_1^\dagger$ into progressions modulo $k_1^*l_2$
\begin{align*}
&{\sum}_{4,m_{n+1}}\leq \\
\leq& Y^{-n-1}\sum_{k_1^*l_2\leq X^{\frac{1}{8}}} \sum_{\bmh{m}}J(\bm{m},Y)\sum_{Y/2k_1^*l_2<k_1^\dagger\leq Y/k_1^*l_2}|Q(\overline{k_1^\dagger}\bm{m},k_1^*l_2)||Q(\overline{k_1^*l_2}\bm{m},k_1^\dagger)|\\
\leq& Y^{-n-1} \sum_{k_1^*l_2\leq X^{\frac{1}{8}}}\sum_{\substack{0\leq \bmh{b}<k_1^*l_2\\0\leq c<k_1^*l_2;(c,k_1^*l_2)=1}}\sum_{\bmh{m}\equiv \bmh{b} (k_1^*l_2)}J(\bm{m},Y)\\
\times&\sum_{\substack{k_1^\dagger\leq Y/k_1^*l_2\\k_1^\dagger\equiv c (k_1^*l_2)}}|Q(\overline{k_1^\dagger}\bm{m},k_1^*l_2)||Q(\overline{k_1^*l_2}\bm{m},k_1^\dagger)|.
\end{align*}
By Lemma \ref{primecase} we have
\begin{align*}
{\sum}_{4,m_{n+1}}\ll& Y^{-n-1}\sum_{k_1^*l_2\leq X^{\frac{1}{8}}}\sum_{\substack{0\leq \bmh{b}<k_1^*l_2\\0\leq c<k_1^*l_2;(c,k_1^*l_2)=1}}|Q((\bmh{b},\overline{c}m_{n+1}),k_1^*l_2)|\sum_{\bmh{m}\equiv \bmh{b} (k_1^*l_2)} J(\bm{m},Y)\\
\times& \sum_{\substack{k_1^\dagger\leq Y/k_1^*l_2\\k_1^\dagger\equiv c (k_1^*l_2)}}A_{2}^{\omega(k_1^{\dagger})}{k_1^{\dagger}}^{\frac{1}{2}(n+2)}(m_{n+1},k_1^{\dagger})^{\frac{1}{2}}\Bigr).
\end{align*}
The innermost sum becomes independent of $\bmh{m}$ and we apply \ref{HoL23} with $l=k_1^*l_2$ to it and the first part of Lemma \ref{Jkongruentsumme} to the sum over $\bmh{m}$. Similar to \cite[(156)]{h9} this gives us for any fixed $M>0$
\begin{align*}
&{\sum}_{4,m_{n+1}}\ll \\ \ll& Y^{-n-1}\sum_{k_1^*l_2\leq X^{1/8}}\frac{Y^{\frac{1}{2}n-\frac{1}{2}}N^{\frac{1}{2}n+\frac{1}{2}}}{X^{n}(k_1^*l_2)^{n}}\log^{n+1}\Bigl(\frac{N}{Y}\Bigr)\Bigl(\frac{1}{|m_{n+1}|+1}\Bigr)^M\\
&\times\frac{Y^{\frac{1}{2}n+2}(\log \log X)^{A_{4}}\sigma_{-\frac{1}{4}}(m_{n+1})}{\phi(k_1^*l_2)(k_1^*l_2)^{\frac{1}{2}n+2}\log X}\sum_{\substack{0\leq \bmh{b}<k_1^*l_2\\0\leq c<k_1^*l_2;(c,k_1^*l_2)=1}}|Q((\bmh{b},\overline{c}m_{n+1}),k_1^*l_2)|\\
\ll& \frac{\sigma_{-\frac{1}{4}}(m_{n+1})Y^{\frac{1}{2}}N^{\frac{1}{2}n+\frac{1}{2}}(\log \log X)^{A_{4}}}{X^{n} \log X}\log^{n+1}\Bigl(\frac{N}{Y}\Bigr)\Bigl(\frac{1}{|m_{n+1}|+1}\Bigr)^M\sum_{l\leq X}\frac{E(l,m_{n+1})}{l^{\frac{3}{2}(n+2)}}.
\end{align*}
We are now in the same situation as before and can apply our bound for $E(w,1)$ in Lemma \ref{KfurE}. A basic calculation using the Euler product and a Mertens formula gives us now
\begin{align*}
\sum_{l\leq X}\frac{E(l,m_{n+1})}{l^{\frac{3}{2}(n+2)}}\leq \prod_{p\leq X}\Bigl(1+\sum_{\alpha\geq 1}\frac{E(p^\alpha,m_n)}{p^{\frac{3}{2}(n+2)\alpha}} \Bigr)=O(\sigma_{-1}(m_{n+1})\log^{1-\delta}X).
\end{align*}
So overall for ${\sum}_{4,m_{n+1}}$ we get 
\begin{align*}
{\sum}_{4,m_{n+1}}=O\Bigl(\frac{\sigma_{-\frac{1}{8}}(m_{n+1})Y^{\frac{1}{2}}N^{\frac{1}{2}n+\frac{1}{2}}}{X^{n}\log^\delta X}\Bigl(\frac{1}{1+|m_{n+1}|}\Bigr)^M\log^{n+1}\Bigl(\frac{N}{Y} \Bigr) \Bigr).
\end{align*}

For ${\sum}_{5,m_{n+1}}$ we start with the same transformation as for ${\sum}_{4,m_{n+1}}$. We have
\begin{align*}
&{\sum}_{5,m_{n+1}} \ll \\
\ll &Y^{-n-1}\sum_{l_2\leq X^{\frac{1}{16}}}\sum_{\substack{0\leq \bmh{b}<l_2\\0\leq c<l_2;(c,l_2)=1}}|Q((\bmh{b},\overline{c}m_{n+1}),l_2)|\sum_{\bmh{m}\equiv \bmh{b} (l_2)} J(\bm{m},Y)\\
&\times \sum_{\substack{k_1\leq Y/l_2\\ \substack{k_1\equiv c (l_2)\\ \omega(k_1)>A_{5}\log\log X}}}A_{2}^{\omega(k_1)}{k_1}^{\frac{1}{2}(n+2)}(m_{n+1},k_1)^{\frac{1}{2}}\\
\ll&Y^{-n-1}\sum_{l_2\leq X^{\frac{1}{16}}}\frac{Y^{\frac{1}{2}n-\frac{1}{2}}N^{\frac{1}{2}n+\frac{1}{2}}}{X^{n}l_2^{n}}\Bigl(\frac{1}{|m_{n+1}|+1}\Bigr)^M\log^{n+1}\Bigl(\frac{N}{Y}\Bigr)  \\
&\times \sum_{\substack{0\leq \bmh{b}<l_2\\0\leq c<l_2;(c,l_2)=1}}|Q((\bmh{b},\overline{c}m_{n+1}),l_2)| \sum_{\substack{k_1\leq Y/l_2\\ \substack{k_1\equiv c (l_2)\\ \omega(k_1)>A_{5}\log\log X}}}A_{2}^{\omega(k_1)}{k_1}^{\frac{1}{2}(n+2)}(m_{n+1},k_1)^{\frac{1}{2}}.
\end{align*}
By Hooley's argument leading to \cite[(159)]{h9} the innermost sum is 
\begin{align*}
O\Bigl(\frac{\sigma_{-\frac{1}{8}}(m_{n+1})Y^{\frac{1}{2}n+\frac{3}{2}}}{{l_2}^{\frac{1}{2}n+3} \log X}\Bigr)
\end{align*}
and so we get 
\begin{align*}
{\sum}_{5,m_{n+1}}&\ll \frac{\sigma_{-1/8}(m_{n+1})Y^{\frac{1}{2}}N^{\frac{1}{2}n+\frac{1}{2}}}{X^{n}\log X}\Bigl(\frac{1}{|m_{n+1}|+1}\Bigr)^M\log^{n+1}\Bigl(\frac{N}{Y} \Bigr) \sum_{l_2\leq X}\frac{E(l_2,m_{n+1})}{l^{\frac{3}{2}n+3}} \\
&\ll \frac{\sigma_{-1/8}(m_{n+1})Y^{\frac{1}{2}}N^{\frac{1}{2}n+\frac{1}{2}}}{X^{n}\log X}\Bigl(\frac{1}{|m_{n+1}|+1}\Bigr)^M\log^{n+1}\Bigl(\frac{N}{Y} \Bigr) .
\end{align*}

Turning our attention to ${\sum}_{6,m_{n+1}}$ we now use Lemma \ref{primesquare} instead of Lemma \ref{primecase}. Afterwards applying Lemma \ref{Jsum} still with the condition $Y\geq X^{5/4}$ in mind we have
\begin{align*}
{\sum}_{6,m_{n+1}}=&Y^{-n-1}\sum_{\substack{\frac{1}{2}Y<k_1k_2l_3\leq Y\\k_2l_3\geq X^{\frac{1}{16}}}}\sum_{\bmh{m}}|Q(\overline{l_3}\bm{m},k_1k_2)||Q(\overline{k_1k_2}\bm{m},l_3)|J(\bm{m},Y)\\
\ll& X^{\epsilon}Y^{-n-1}\sum_{\substack{k_1k_2l_3\leq Y\\k_2l_3\geq X^{\frac{1}{16}}}}(m_n^2,k_1k_2)^{\frac{1}{2}}(k_1k_2)^{\frac{1}{2}n+1}\sum_{\bmh{m}}|Q(\overline{k_1k_2}\bm{m},l_3)|J(\bm{m},Y) \\
\ll& \Bigl(\frac{1}{1+|m_{n+1}|}\Bigr)^MX^\epsilon Y^{-n-1}\sum_{k_1k_2l_3\leq Y}(m_n^2,k_1k_2)^{\frac{1}{2}}(k_1k_2)^{\frac{1}{2}n+1}l_3^{\frac{5}{6}n+\frac{3}{2}} \\
&+\Bigl(\frac{1}{1+|m_{n+1}|}\Bigr)^M\frac{X^\epsilon N^{\frac{1}{2}n+\frac{1}{2}}}{Y^{\frac{1}{2}n+\frac{3}{2} }X^{n}}\sum_{\substack{k_1k_2l_3\leq Y\\k_2l_3\geq X^{\frac{1}{16}}}}(m_{n+1}^2,k_1k_2)^{\frac{1}{2}}(k_1k_2)^{\frac{1}{2}n+1}l_3^{\frac{1}{2}n+\frac{3}{2}}.
\end{align*}
The remaining sums are variants of objects Hooley dealt with and following his calculations around \cite[(161)]{h9} we get for the first sum
\begin{align*}
\sum_{k_2l_3\leq Y}(m_{n+1}^2,k_2)^{\frac{1}{2}}&k_2^{\frac{1}{2}n+1}l_3^{\frac{5}{6}n+\frac{3}{2}}\sum_{k_1\leq Y/k_2l_3}(m_{n+1},k_1)^{\frac{1}{2}}k_1^{\frac{1}{2}n+1}\\
&=O\Bigl(Y^{\frac{1}{2}n+2}\sigma_{-\frac{1}{2}}(m_{n+1})\sum_{k_2l_3\leq Y}\frac{(r^2,k_2)^{\frac{1}{2}}l_3^{\frac{1}{3}n-\frac{1}{2}}}{k_2} \Bigr)\\
&=O\Bigl(Y^{\frac{1}{2}n+2}\sigma_{-\frac{1}{2}}(m_{n+1})\sum_{k_2}\frac{(r^2,k_2)^{\frac{1}{2}}}{k_2}\sum_{l_3\leq Y}l_3^{\frac{1}{3}n-\frac{1}{2}} \Bigr)\\
&=O\Bigl(Y^{\frac{5}{6}n+\frac{11}{6}}\sigma_{-\frac{1}{2}}(m_{n+1})\sigma_{-1}(m_{n+1}) \Bigr).
\end{align*}
The second sum can be similarly bounded by 
\begin{align*}
Y^{\frac{1}{2}n+2}X^{-\frac{1}{96}}\sigma_{-\frac{1}{2}}(m_{n+1})\sigma_{-\frac{2}{3}}(m_{n+1}).
\end{align*}
The overall contribution of ${\sum}_{6,m_{n+1}}$ is consequently
\begin{align*}
{\sum}_{6,m_{n+1}}=&O\Bigl(\Bigl(\frac{1}{1+|m_{n+1}|}\Bigr)^M\sigma_{-\frac{1}{8}}(m_{n+1})X^\epsilon Y^{-\frac{1}{6}n+\frac{5}{6}} \Bigr)\\
&+O\Bigl(\Bigl(\frac{1}{1+|m_{n+1}|}\Bigr)^M \sigma_{-\frac{1}{8}}(m_{n+1}) X^{-n-\frac{1}{96}+\epsilon}Y^{\frac{1}{2}}N^{\frac{1}{2}n+\frac{1}{2}}\Bigr).
\end{align*}

The estimates for ${\sum}_{4,m_{n+1}}$, ${\sum}_{5,m_{n+1}}$, and ${\sum}_{6,m_{n+1}}$ are, except of the terms depending on $m_{n+1}$, respectively the same as the ones for $\sum_{1}$, $\sum_{2}$, and $\sum_{3}$. For $Y\geq X^{5/4}$ their contribution hence is

\begin{align*}
\sum_{X^{5/4}\leq Y_j \leq A_1N}P_3(Y,m_{n+1})\ll \Bigl(\frac{1}{1+|m_{n+1}|}\Bigr)^M \sigma_{-\frac{1}{8}}(m_{n+1}) (\log X)^{-\delta}.
\end{align*}
The contribution of $Y<X^{5/4}$ can be calculated by changing the above estimates the same way as we did in the case of $P_2(Y)$. Overall we get by choosing $M=3$

\begin{align*}
\Upsilon_3(X)&\ll\sum_{m_{n+1}\neq 0} \Bigl(\frac{1}{1+|m_{n+1}|} \Bigr)^3\sigma_{-\frac{1}{8}}(m_{n+1})(\log X)^{-\delta}\\
&\ll (\log X)^{-\delta} \sum_{m_{n+1}\neq 0}|m_{n+1}|^{-2}\\
&\ll (\log X)^{-\delta}.
\end{align*}

\subsection{$\Upsilon_4$}\label{secup4}
To estimate $P_4(Y)$ we start with the transformation \eqref{Ho139}. As now $$F_1(\bmh{m})=F_2(\bmh{m})=m_{n+1}=0,$$ we can no longer apply the better cases of Lemma \ref{primecase} and \ref{primesquare}. With a later determined choice of $1\leq \xi\leq A_1 X^{\frac{3}{2}}$ these now give

\begin{align*}
&P_4(Y)\ll \\
\ll& X^\epsilon Y^{-n-1}\sum_{l_3\leq Y}\sum_{\substack{\bmh{m}\neq 0\\ F_1(\bmh{m})=F_2(\bmh{m})=0}}|Q((\bmh{m},0),l_3)|J((\bmh{m},0),Y)\sum_{k_1 k_2\leq Y/l_3}(k_1k_2)^{\frac{1}{2}n+\frac{3}{2}}\\
\ll& X^\epsilon Y^{-\frac{1}{2}n+\frac{3}{2}}\sum_{l_3\leq Y}\frac{1}{l_3^{\frac{1}{2}n+\frac{5}{2}}}\sum_{\substack{\bmh{m}\neq 0\\ F_1(\bmh{m})=F_2(\bmh{m})=0}}|Q((\bmh{m},0),l_3)|J((\bmh{m},0),Y)\\
\ll& X^\epsilon Y^{-\frac{1}{2}n+\frac{3}{2}}\Bigl(\sum_{l_3\leq \xi}\sum_{\substack{\bmh{m}\neq 0\\ F_1(\bmh{m})=F_2(\bmh{m})=0}}+\sum_{\xi\leq l_3\leq Y}\sum_{\bmh{m}\neq 0} \Bigr)\\
=&P_4^{'}(Y)+P_4^{''}(Y),
\end{align*}
say. The optimal value of $\xi$ will depend on $Y$ and $X$ and may exceed $Y$, in which case $P_4^{''}(Y)=0$. To bound $P_4^{'}(Y)$ we write $l_3=k_3l_4$ and apply the first statement of Lemma \ref{JsumD0} to get
\begin{align*}
P_4^{'}(Y)&\ll \frac{X^\epsilon\min(X,Y)^{\frac{1}{2}n-\frac{1}{2}}}{X^{\frac{1}{2}n-1}Y^{\frac{1}{2}n-\frac{3}{2}}}\sum_{k_3l_4\leq \xi}k_3^{\frac{1}{3}n+\frac{1}{3}-\frac{3}{2}}l_4^{\frac{1}{2}n-\frac{3}{2}}+\frac{X^\epsilon Y N^{\frac{1}{2}n-\frac{3}{2}}}{X^{n-2}}\sum_{l_3\leq \xi}l_3^{-\frac{1}{2}}\\
&\ll \frac{X^\epsilon\min(X,Y)^{\frac{1}{2}n-\frac{1}{2}}}{X^{\frac{1}{2}n-\frac{1}{2}}Y^{\frac{1}{2}n-\frac{3}{2}}}\sum_{k_3l_4\leq \xi}k_3^{\frac{1}{3}n+\frac{1}{3}-\frac{3}{2}}l_4^{\frac{1}{2}n-\frac{3}{2}}+\frac{X^\epsilon Y N^{\frac{1}{2}n-\frac{3}{2}}}{X^{n-2}}.
\end{align*}
We now fix $n$ as its intended value $6$ and get for the remaining sum
\begin{align*}
\sum_{k_3l_4\leq \xi}k_3^{\frac{5}{6}}l_4^{\frac{3}{2}}=\sum_{k_3\leq \xi}k_3^{\frac{5}{6}}\sum_{l_4\leq \xi/k_3}l_4^{\frac{3}{2}}\leq \xi^{\frac{7}{4}}\sum_{k_3\leq \xi}k_3^{-\frac{11}{12}}\ll \xi^{\frac{7}{4}}.
\end{align*}
Therefore we estimate
\begin{align*}
P_4^{'}(Y)\ll\frac{X^\epsilon \min(X,Y)^{\frac{5}{2}}\xi^{\frac{7}{4}}}{X^{\frac{5}{2}}Y^{\frac{3}{2}}}+ X^{-\frac{7}{4}+\epsilon}Y.
\end{align*}
Similarly, now employing Lemma \ref{Jsum} instead of Lemma \ref{JsumD0}, we have

\begin{align*}
P_4^{''}(Y)&\ll \frac{X^\epsilon\min (X,Y)^{\frac{1}{2}n-\frac{1}{2}}}{X^{\frac{1}{2}n-\frac{1}{2}}Y^{\frac{1}{2}n-\frac{3}{2}}}\sum_{l_3 \leq Y}l_3^{\frac{1}{3}n-1}+\frac{X^\epsilon Y N^{\frac{1}{2}n+\frac{1}{2}}}{X^{n}}\sum_{l_3 \geq \xi}\frac{1}{l_3}\\
&\ll \frac{X^\epsilon\min (X,Y)^{\frac{1}{2}n-\frac{1}{2}}}{X^{\frac{1}{2}n-\frac{1}{2}}Y^{\frac{1}{6}n-\frac{5}{6}}}+\frac{X^\epsilon Y N^{\frac{1}{2}n+\frac{1}{2}}}{X^{n}\xi^{\frac{2}{3}}}\\
&\ll \frac{X^\epsilon\min (X,Y)^{\frac{5}{2}}}{X^{\frac{5}{2}}Y^{\frac{1}{6}}}+\frac{X^\epsilon Y X^{-\frac{3}{4}}}{\xi^{\frac{2}{3}}}.
\end{align*}
This means that the optimal value of $\xi$ is determined by

\begin{align*}
\frac{\min(X,Y)^{\frac{5}{2}}\xi^{\frac{7}{4}}}{X^{\frac{5}{2}}Y^{\frac{3}{2}}}=\frac{ Y X^{-\frac{3}{4}}}{\xi^{\frac{2}{3}}}.
\end{align*}
This results in
\begin{align*}
\xi=\begin{cases}X^{-\frac{9}{29}}Y^{\frac{30}{29}} \text{, if } Y\geq X \\ X^{\frac{21}{29}}\text{, if } Y<X  \end{cases},
\end{align*}
which falls within the initial terms of reference. For $Y\geq X$ we thus have
\begin{align*}
P_4(Y)\ll  X^{-\frac{7}{4}+\epsilon}Y + X^\epsilon Y^{-\frac{1}{6}} + X^{-\frac{63}{116}+\epsilon}Y^{\frac{9}{29}}.
\end{align*}
If $Y<X$ then we have
\begin{align*}
P_4(Y)\ll X^{-\frac{7}{4}+\epsilon}Y+X^{-\frac{5}{2}+\epsilon}Y^{\frac{7}{3}}+X^{-\frac{143}{116}+\epsilon}Y.
\end{align*}
Putting the results together and recalling $Y\leq A_1N=A_1 X^{\frac{3}{2}}$ we arrive at
\begin{align*}
P_4(Y)&\ll X^{-\frac{1}{4}+\epsilon}+X^{-\frac{1}{6}+\epsilon}+X^{-\frac{1}{6}+\epsilon}+X^{-\frac{27}{116}+\epsilon}+X^{-\frac{9}{116}+\epsilon}\\
&\ll X^{-\frac{9}{116}+\epsilon}.
\end{align*}
Overall we get for $n\geq 6$
\begin{align*}
\Upsilon_4(X)=\sum_{1\leq Y_j\leq A_1N}X^{-\frac{9}{116}+\epsilon}\ll X^{-\frac{9}{116}+\epsilon}. 
\end{align*}
\newpage
\subsection{$\Upsilon_5$}\label{secup5}
The estimation of $P_5(Y)$ is closely related to section 48 of the third of Hooley's Papers on cubic forms \cite{h93}.
We start by a similar transformation as for $P_4(Y)$ and get
\begin{align*}
&P_5(Y)\ll \\
\ll& X^\epsilon Y^{-n-1}\!\!\sum_{\substack{\substack{l_3\leq Y\text{, }\bmh{m}\neq 0 }\\ \substack{ F_1(\bmh{m})=0 \\ F_2(\bmh{m})\neq 0}}}|Q((\bmh{m},0),l_3)|J((\bmh{m},0),Y)\sum_{k_1 k_2\leq Y/l_3}k_1^{\frac{1}{2}n+1}(k_1,F_2(\bmh{m}))^{\frac{1}{2}} k_2^{\frac{1}{2}n+\frac{3}{2}}\\
\ll& X^\epsilon Y^{-\frac{1}{2}n+1}\sum_{l_3\leq Y}\frac{1}{l_3^{\frac{1}{2}n+2}}\sum_{\substack{\bmh{m}\neq 0\\ F_1(\bmh{m})=0}}|Q((\bmh{m},0),l_3)|J((\bmh{m},0),Y)\\
\ll& X^\epsilon Y^{-\frac{1}{2}n+1}\Bigl(\sum_{l_3\leq \xi}\sum_{\substack{\bmh{m}\neq 0\\ F_1(\bmh{m})=0}}+\sum_{\xi\leq l_3\leq Y}\sum_{\bmh{m}\neq 0} \Bigr)\\
=&P_5^{'}(Y)+P_5^{''}(Y),
\end{align*}
say. We apply the second statement of Lemma 6.4 on $P_5^{'}(Y)$ and get
\begin{align*}
P_5^{'}(Y)&\ll \frac{X^\epsilon \min(X,Y)^{\frac{1}{2}n-\frac{1}{2}}}{X^{\frac{1}{2}n-\frac{1}{2}}Y^{\frac{1}{2}n-1}}\sum_{k_3l_4\leq \xi}k_3^{\frac{1}{3}n+\frac{1}{3}-1}l_4^{\frac{1}{2}n-1}+\frac{Y^{\frac{1}{2}}N^{\frac{1}{2}n-\frac{1}{2}}}{X^{n-1-\epsilon}}\sum_{l_3\leq \xi}l_3^{-\frac{1}{3}}\\
&\ll \frac{X^\epsilon \min(X,Y)^{\frac{1}{2}n-\frac{1}{2}}}{X^{\frac{1}{2}n-\frac{1}{2}}Y^{\frac{1}{2}n-1}}\sum_{k_3l_4\leq \xi}k_3^{\frac{1}{3}n+\frac{1}{3}-1}l_4^{\frac{1}{2}n-1}+\frac{Y^{\frac{1}{2}}N^{\frac{1}{2}n-\frac{1}{2}}}{X^{n-1-\epsilon}}.
\end{align*}
Continuing as before, we set $n$ to $6$ and get for the sum over $k_3l_4$ the estimate
\begin{align*}
\sum_{k_3l_4\leq \xi}k_3^{\frac{4}{3}}l_4^2&=\sum_{k_3\leq \xi}k_3^{\frac{4}{3}}\sum_{l_4\leq \xi/k_3}l_4^2\\
&\ll \xi^{\frac{9}{4}}\sum_{k_3\leq \xi}k_3^{-\frac{3}{4}}\\
&\ll \xi^{\frac{9}{4}}.
\end{align*}
We recall $N=A_1X^{\frac{3}{2}}$ and conclude  
\begin{align*}
P_5^{'}(Y)&\ll \frac{X^\epsilon \min(X,Y)^{\frac{5}{2}}\xi^{\frac{9}{4}}}{X^{\frac{5}{2}}Y^2}+Y^{\frac{1}{2}}X^{-\frac{5}{4}+\epsilon}.
\end{align*}
We apply Lemma 6.3 on $P_5^{''}(Y)$ set $n=6$ and get
\begin{align*}
P_5^{''}(Y)&\ll \frac{X^\epsilon\min (X,Y)^{\frac{1}{2}n-\frac{1}{2}}}{X^{\frac{1}{2}n-\frac{1}{2}}Y^{\frac{1}{2}n-1}}\sum_{l_3 \leq Y}l_3^{\frac{1}{3}n-\frac{1}{2}}+\frac{X^\epsilon Y^{\frac{1}{2}} N^{\frac{1}{2}n+\frac{1}{2}}}{X^{n}}\sum_{l_3 \geq \xi}\frac{1}{l_3^{\frac{1}{2}}}\\
&\ll \frac{X^\epsilon\min (X,Y)^{\frac{1}{2}n-\frac{1}{2}}}{X^{\frac{1}{2}n-\frac{1}{2}}Y^{\frac{1}{6}n-\frac{5}{6}}}+\frac{X^\epsilon Y^{\frac{1}{2}} N^{\frac{1}{2}n+\frac{1}{2}}}{X^{n}\xi^{\frac{1}{6}}}\\
&\ll \frac{X^\epsilon\min (X,Y)^{\frac{5}{2}}}{X^{\frac{5}{2}}Y^{\frac{1}{6}}}+\frac{X^\epsilon Y^{\frac{1}{2}} X^{-\frac{3}{4}}}{\xi^{\frac{1}{6}}}.
\end{align*}
Consequently the optimal value of $\xi$ is now determined by the equation
\begin{align*}
\frac{\min(X,Y)^{\frac{5}{2}}\xi^{\frac{9}{4}}}{X^{\frac{5}{2}}Y^{2}}=\frac{ Y^{\frac{1}{2}} X^{-\frac{3}{4}}}{\xi^{\frac{1}{6}}}.
\end{align*}
As in \cite{h93} and \cite{h9} the optimal values of $\xi$ is the same as in the previous subsection. We again have
\begin{align*}
\xi=\begin{cases}X^{-\frac{9}{29}}Y^{\frac{30}{29}} \text{, if } Y\geq X \\ X^{\frac{21}{29}}\text{, if } Y<X  \end{cases}.
\end{align*}
Depending on weather $Y\geq X$ or $Y<X$ we get
\begin{align*}
P_5(Y)&\ll X^{-\frac{5}{4}+\epsilon}Y^{\frac{1}{2}}+Y^{-\frac{1}{6}+\epsilon}+Y^{\frac{19}{58}}X^{-\frac{81}{116}}\\
\end{align*}
or
\begin{align*}
P_5(Y)\ll X^{-\frac{5}{4}+\epsilon}Y^{\frac{1}{2}}+Y^{\frac{7}{3}}X^{-\frac{5}{2}}+Y^{\frac{1}{2}}X^{-\frac{101}{116}+\epsilon}.
\end{align*}
So overall
\begin{align*}
\Upsilon_5(X)\ll X^{-\frac{1}{6}+\epsilon}.
\end{align*}

\subsection{$\Upsilon_1$, Singular Series, Singular Integral, and Conclusion}

It remains to evaluate the proposed main term $\Upsilon_1(X)$ that is given by
\begin{align*}
\Upsilon_1(X)=\sum_{k\leq A_1N}k^{-n-1}Q(\bm 0,k)I_k(\bm 0).
\end{align*}

Associated with it is the singular series 
\begin{align*}
\mathfrak{S}=\sum_{k=1}^\infty k^{-n}Q(\bm 0,k). 
\end{align*}
This series is absolutely convergent, if $n\geq 6$. Indeed, writing $k=k_1k_2l_3$ as before, we can use \ref{primecase}, \ref{primesquare}, and \ref{6.1} to get
\begin{align*}
 \sum_{L< k\leq 2L}k^{-n-1}|Q(\bm 0,k)|&\ll \sum_{L< k_1k_2l_3\leq 2L}(k_1k_2l_3)^{-n-1}(k_1k_2l_3)^{\frac{5}{6}(n+1)+\epsilon}l_3^{\frac{2}{3}} \\
 &\leq L^{-\frac{1}{6}(n+1)+\epsilon}\sum_{l_3\leq L}l_3^{\frac{2}{3}}\\
 &\leq L^{-\frac{1}{6}(n+1)+1+\epsilon}.
\end{align*}

By the usual theory the singular series is related to the existence of $p$-adic solutions of the equation in question. For primitive $C$ it is easy to show that these exist and so we have $\mathfrak{S}>0$. In that case the reduction $C$ modulo $p$ is never identically zero. If, after changing the indices, we have $C(x_1,0,\ldots ,0)\equiv a x_1^3 (p)$ for some $p\nmid a$, then surely \begin{align*}
C(x_1,0,\ldots ,0)\equiv x_{n+1}^2 (p)
\end{align*}
has a nontrivial and thus nonsingular solution, which lifts to a $p$-adic one by Hensel's Lemma.
The next possible case is that we have $C(x_1,x_2,0,\ldots ,0)\equiv a x_1^2 x_2+bx_1x_2^2 (p)$ with again $p\nmid a$. For odd $p$ there is again a solution having $p\nmid x_n$ which lifts. For $p=2$ we choose $x_1=x_n=1$ and $x_2=0$.
The last case is that $C(x_1,x_2,x_3,\ldots ,0)\equiv a x_1 x_2 x_3 (p)$, which is trivial.

Next, looking at $I_k$, we recall that we chose $\bm{a}=(\bmh{a},0)$, such that $\bmh{a}$ is a (nonsingular) zero of $C$. This means the remaining Integral
\begin{align*}
I_k(\bm{0})=\int_{\mathbb{R}^n}\Gamma(\bm{t}-\bm{a})h\Bigl(\frac{k}{N},f(\bm{t})\Bigr)d\bm{t}
\end{align*}
can be treated in the same way as the analogous object of \cite{h8}. After replacing the variables of integration we can apply Lemma 9 of \cite{hb1}. This gives us
$$I_k(0)=\mathcal{J}+O(\frac{k^{1/12}}{N^{1/12}})=O(1),$$
where $\mathcal{J}>0$ by the implicit function theorem as $\bm{a}$ is a zero of $f$.
Putting together these results we have for $n\geq 6$
\begin{align*}
\Upsilon_1(X)&=\mathcal{J}\mathfrak{S}+O(\sum_{k>A_1N}k^{-n-1}Q(\bm 0,k))+O(N^{-1/12}\sum_{k\leq A_1N}k^{-n-1+\frac{1}{12}}Q(\bm 0,k))\\
&=\mathcal{J}(0)\mathfrak{S}+O(N^{-\frac{1}{6}+\epsilon})+O(N^{-\frac{1}{12}+\epsilon})\\
&=\mathcal{J}\mathfrak{S}+O(X^{-\frac{1}{4}})
\end{align*}

Combining the results of the previous subsection with \eqref{UPdef} and \eqref{c_N}, we get for $n\geq 6$ 
\begin{align*}
\Upsilon(X)&=c_N X^{n-\frac{5}{2}}\Bigl(\Upsilon_1(X)+\Upsilon_2(X)+\Upsilon_3(X)+\Upsilon_4(X) \Bigr)\\
&=X^{n-\frac{5}{2}}\mathcal{J}\mathfrak{S}\Bigl(1+O( \log^{-\delta} X)\Bigr).
\end{align*}
This completes the proof of Theorem \ref{Thm1.1}.

\end{document}